 \documentclass[a4paper,11pt]{article}
\usepackage{amsmath,amsthm, authblk}
\usepackage{amssymb,latexsym}
\usepackage{graphicx,subfig}
\usepackage{epsfig}
\usepackage{color}

\newtheorem{thm}{Theorem}[section]

\newtheorem{lma}[thm]{Lemma}
\newtheorem{cor}[thm]{Corollary}
\newtheorem{prp}[thm]{Proposition}

\newtheorem{clm}[thm]{Claim}
\newtheorem{rmk}[thm]{Remark}

\newtheorem*{alg}{Algorithm for finding $F$-factors}

\newtheorem{qns}[thm]{Question}

\numberwithin{equation}{section}
\def\eps{\varepsilon}
\DeclareMathOperator{\ex}{ex}

\usepackage{lineno}
\newcommand*\patchAmsMathEnvironmentForLineno[1]{%
  \expandafter\let\csname old#1\expandafter\endcsname\csname #1\endcsname
  \expandafter\let\csname oldend#1\expandafter\endcsname\csname end#1\endcsname
  \renewenvironment{#1}%
     {\linenomath\csname old#1\endcsname}%
     {\csname oldend#1\endcsname\endlinenomath}}% 
\newcommand*\patchBothAmsMathEnvironmentsForLineno[1]{%
  \patchAmsMathEnvironmentForLineno{#1}%
  \patchAmsMathEnvironmentForLineno{#1*}}%
\AtBeginDocument{%
\patchBothAmsMathEnvironmentsForLineno{equation}%
\patchBothAmsMathEnvironmentsForLineno{align}%
\patchBothAmsMathEnvironmentsForLineno{flalign}%
\patchBothAmsMathEnvironmentsForLineno{alignat}%
\patchBothAmsMathEnvironmentsForLineno{gather}%
\patchBothAmsMathEnvironmentsForLineno{multline}%
}

\title{$F$-factors in hypergraphs via absorption}
\author{Allan Lo
\footnote{This author was supported by the ERC, grant no.~258345.}
}
\affil{School of Mathematics, University of Birmingham, \\Birmingham, B15 2TT, UK\\
\texttt{s.a.lo@bham.ac.uk}}
\author{Klas Markstr{\"o}m}
\affil{Department of Mathematics and Mathematical Statistics,\\ Ume\r{a} University, S-901 87 Ume\r{a}, Sweden\\
\texttt{klas.markstrom@math.umu.se}}

\begin{document}
\maketitle

\abstract{
Given integers $n \ge k >l \ge 1$ and a $k$-graph $F$ with $|V(F)|$ divisible by~$n$, define $t_l^k(n,F)$ to be the smallest integer $d$ such that every $k$-graph $H$ of order $n$ with minimum $l$-degree $\delta_l(H) \ge d $ contains an $F$-factor.
A classical theorem of Hajnal and Szemer\'{e}di~\cite{MR0297607} implies that $t^2_1(n,K_t) = (1-1/t)n$ for integers~$t$.
For $k \ge 3$, $t^k_{k-1}(n,K_k^k)$ (the $\delta_{k-1}(H)$ threshold for perfect matchings) has been determined by K\"{u}hn and Osthus~\cite{MR2207573} (asymptotically) and R\"{o}dl, Ruci\'{n}ski and Szemer\'{e}di~\cite{MR2500161} (exactly) for large $n$.

In this paper, we generalise the absorption technique of  R\"{o}dl, Ruci\'{n}ski and Szemer\'{e}di~\cite{MR2500161} to $F$-factors.
We determine the asymptotic values of $t^k_1(n,K_k^k(m))$ for $k = 3,4$ and $m \ge 1$.
In addition, we show that for $t>k = 3$ and $\gamma >0$, $ t^3_{2}(n,K_t^3) \le \left( 1- \frac{2}{t^2-3t+4} + \gamma \right) n$ provided $n$ is large and $t | n$.
We also bound $t^3_{2}(n,K_t^3)$ from below.
In particular, we deduce that $t^3_2(n,K_4^3) = (3/4+o(1))n$ answering a question of Pikhurko~\cite{MR2438870}.
In addition, we prove that $t^k_{k-1}(n,K_t^k) \le (1- \binom{t-1}{k-1}^{-1} + \gamma)n$ for $\gamma >0$, $k \ge 6$ and $t \ge (3+ \sqrt5)k/2$  provided $n$ is large and $t | n$.
}

% \linenumbers

\section{Introduction} \label{sec:introduction}

Given graphs $G$ and $F$, an \emph{$F$-factor} (or \emph{perfect $F$-tiling}) in~$G$ is a spanning subgraph consisting of vertex-disjoint copies of~$F$.
Clearly, if $G$ contains an $F$-factor then $|V(F)|$ divides~$|V(G)|$.
Given an integer $t$ and a graph $G$ of order $n$ with $t|n$, we would like to know the minimum degree threshold that guarantees a $K_t$-factor in~$G$.
Note that the minimum degree must be at least $(t-1)n/t$ by considering a complete $t$-partite graph with partition classes $V_1, \dots, V_t$ with $|V_1| = n/t-1$, $|V_t| = n/t+1$ and $|V_i| = n/t$ for $1<i<t$.
In fact, $\delta(G) \ge (t-1)n/t$ suffices.
For $t=2$, this can be easily verified using Dirac's Theorem~\cite{MR0047308}.
Corr\'{a}di and Hajnal~\cite{MR0200185} proved the case for $t=3$.
All remaining cases $t \ge 4$ can be verified by a classical theorem of Hajnal and Szemer\'{e}di~\cite{MR0297607}.
In this paper, we ask the same question for hypergraphs. 

We often write $l$-sets for $l$-element sets.
Given a set $U$, we denote by $\binom{U}l$ the set of all $l$-sets in~$U$.
A \emph{$k$-uniform hypergraph}, \emph{$k$-graph} for short, is a pair $H = (V(H),E(H))$, where $V(H)$ is a finite set of vertices and $E(H) \subseteq \binom{V(H)}{k}$.
Often we write $V$ instead of $V(H)$ when it is clear from the context.
Given a $k$-graph $H$ and an $l$-set $T \in \binom{V}l$, let $\deg(T)$ be the number of $(k-l)$-sets $S \in \binom{V}{k-l}$ such that $S \cup T$ is an edge in~$H$.
Define $\delta_l(H)$ to be the \emph{minimum $l$-degree} of $H$, that is, $\min \deg(T)$ over all $T \in \binom{V}l$.
Note that $\delta_{1}(H)$ is the minimum vertex degree and $\delta_{k-1}(H)$ is the minimum codegree of~$H$.

Analogously, given a hypergraph~$H$ and a family~$\mathcal{F}$ of hypergraphs, an $\mathcal{F}$-factor is a spanning subgraph consisting of vertex-disjoint copies of members of~$\mathcal{F}$.
For a family~$\mathcal{F}$ of $k$-graphs, define $t_l^k(n,\mathcal{F})$ to be the smallest integer $d$ such that every $k$-graph $H$ of order $n$ satisfying $\delta_l(H) \ge d $ contains an $\mathcal{F}$-factor.
Throughout this paper, $\mathcal{F}$ is assumed to be $\{F\}$, so we simply write $F$-factor and $t_l^k(n,F)$.
Note that if $|V(F)|$ does not divide~$n$, then $t^k_l(n,F)$ is not defined.
Thus, we always assume that $|V(F)|$ divides $n$ whenever we talk about $t^k_l(n,F)$.
When $l=k-1$, we simply write $t^k(n,F)$.
Let $t_l^k(n, t)$ denote $t_l^k(n, K_t^k)$, where $K_t^k$ is the complete $k$-graph on $t$ vertices.
Thus, $t^2(n, t) = (t-1)n/t$.

For graphs (that is, 2-graphs) $F$, there is a large body of research on $t^2(n, F)$, for surveys see~\cite{MR2588541,yuster2007combinatorial}. 
However, for $k \ge 3$, only a few values of $t^k_l(n,F)$ are determined for $k$-graphs~$F$.
Note that $K_k^k$ is a single edge, so a $K_k^k$-factor is equivalent to a perfect matching.
K\"{u}hn and Osthus~\cite{MR2207573} showed that $t^k(n,k) = n/2 + O(\sqrt{n \log n})$.
Later, R\"{o}dl, Ruci\'{n}ski and Szemer\'{e}di~\cite{MR2500161} evaluated the exact value of $t^k(n,k)$ using an absorption technique.
H\'{a}n, Person and Schacht~\cite{MR2496914} conjectured that for $1 \le l < k$
\begin{align*}
t_l^k(n, k) \approx \max \left\{ \frac12, 1 - \left(1-\frac1k\right)^{k-l} \right\} \binom{n}k.
\end{align*}
We recommend \cite{rodldirac} for a survey of the recent developments in~$t_l^k(n, k)$.
Pikhurko~\cite{MR2438870} showed that $3n/4-2 \le t^3(n, 4) \le 0.860 n$ and asked whether $ t^3(n, 4) = 3n/4-2$.
For the unique 3-graph $F$ of order~4 with 2 edges, K\"{u}hn and Osthus~\cite{MR2274077} showed that $t^3(n,F) = (1/4+o(1))n$, and the exact value was determined by Czygrinow, DeBiasio and Nagle~\cite{czygrinow2011tiling} for large~$n$.
For the unique 3-graph $K_4^3-e$ of order~4 with 3 edges, the authors~\cite{MR3007146} showed that $t^3(n,K_4^3-e) = (1/2+o(1))n$.
Recently, Kierstead and Mubayi~\cite{kierstead2010toward} proved a  generalisation of the Hajnal--Szemer\'{e}di theorem for 3-graphs and vertex degree, which implies that 
\begin{align*}
t_1^3(n,t) \le \left( 1- \frac{c}{t^2 \log^4(n/t)}  \right) \binom{n}2
\end{align*}
for some constant $c>0$.

One of the key techniques in finding perfect matchings (evaluating $t_l^k(n, k)$) is the absorption technique, which was first introduced by R\"{o}dl, Ruci\'{n}ski and Szemer\'{e}di~\cite{MR2500161}.
Roughly speaking, the absorption technique reduces the task of finding a perfect matching in $H$ to finding a matching covering all but at most $\eps |V(H)|$ vertices for some small $\eps >0$. 
Here, we generalise the absorption technique to $F$-factors.

Let $H$ be a $k$-graph of order $n$.
Given a vertex set $U \subseteq V(H)$, $H[U]$ is the subgraph of~$H$ induced by the vertices of~$U$.
We write $v$ to mean the set $\{v\}$ when it is clear from the context.
Given a $k$-graph~$F$ of order~$t$, an integer $i \ge 1$ and a constant $\eta>0$, we say that a vertex~$x$ is \emph{$(F, i, \eta)$-close} to a vertex~$y$ if there exist $\eta n^{it-1}$ sets~$S \subseteq V(H)$ of size $it-1$ such that $S \cap \{x,y\} = \emptyset$ and both $H[S \cup x]$ and $H[S \cup y]$ contain $F$-factors.
Moreover, $H$ is said to be \emph{$(F, i , \eta)$-closed} if every vertex is $(F, i, \eta)$-close to all other vertices.
We now state the absorption lemma for $F$-factors.

\begin{lma}[Absorption lemma for $F$-factors] \label{lma:absorptionlemma}
Let $t$ and $i$ be positive integers and let $\eta >0$.
Let $F$ be a hypergraph of order~$t$.
Then, there is an integer $n_0 = n_0(t,i,\eta)$ satisfying the following: Suppose that $H$ is an $(F,i,\eta)$-closed hypergraph of order $n \ge n_0$.
Then there exists a vertex subset $U \subseteq V(H)$ of size $|U| \le (\eta/2)^{t} n / (4 it(t-1))$ with $|U| \in t \mathbb{Z}$ such that there exists an $F$-factor in $H[U \cup W]$ for every vertex set $W \subseteq V \setminus U$ of size $|W|  \le (\eta/2)^{2t} n / (32 i^2 t (t-1)^2) $ with $ |W| \in t \mathbb{Z}$.
\end{lma}

Note that in the above lemma $H$ and $F$ are not necessarily $k$-graphs, but we only consider $k$-graphs in this paper.
When we say that $H$ has an almost $F$-factor $\mathcal{T}$, we mean that $\mathcal{T}$ is a set of vertex-disjoint copies of $F$ in $H$ such that  $|V(H) \setminus V(\mathcal{T})| < \eps |V(H)|$ for some small $\eps >0 $.
Equipped with the absorption lemma, we can break down the task of finding an $F$-factor in large hypergraphs $H$ into the following algorithm.
\begin{alg}\hfill
\begin{itemize}
\item[1.] Remove a small set $\mathcal{T}_1$ of vertex-disjoint copies of $F$ from $H$ such that the resultant graph $H_1 = H [V\setminus V(\mathcal{T}_1)]$ is $(F,i,\eta)$-closed for some integer $i$ and constant $\eta>0$.
\item[2.] Find a vertex set $U \subseteq V(H_1)$ satisfying the conditions of the absorption lemma. Set $H_2 = H_1[V(H_1) \setminus U]$.
\item[3.] Show that $H_2$ contains an almost $F$-factor, i.e. a set  $\mathcal{T}_2$ of vertex-disjoint copies of $F$ such that  $|V(H_2) \setminus V(\mathcal{T}_2)| < \eps |V(H_2)|$ for small $\eps >0 $.
\item[4.] Set $W = V(H_2) \setminus V(\mathcal{T}_2)$.
Since $H_1[U \cup W]$ contains an $F$-factor~$\mathcal{T}_3$ by the choice of~$U$, $\mathcal{T}_1 \cup \mathcal{T}_2 \cup \mathcal{T}_3$ is an $F$-factor in $H$.
\end{itemize}
\end{alg}
We now apply the algorithm to various $k$-graphs $F$.
We would like to point out that Steps 1 and 3 of the algorithm require most of the work.

A $k$-graph $H$ is \emph{$k$-partite} with partition $V_1,\dots,V_k$, if $V = V_1 \cup  \dots \cup V_k$ and every edge intersects every $V_i$ in exactly one vertex.
We denote by $K_k^k(m_1, \dots,m_k)$ the complete $k$-partite $k$-graph with parts of sizes $m_1, \dots, m_k$.
If $m = m_i$ for all $1 \le i \le m_k$, we simply write $K_k^k(m)$.
Clearly, $t^k_l(n,K^k_k(m)) \ge t^k_l(n,K^k_k) = t^k_l(n,k)$.
As a simple application of the absorption lemma, we show that $t^k_1(n,K^k_k(m))$ equals $t^k_1(n,k)$ asymptotically for integers $k = 3,4$ and $m\ge 1$. 
\begin{thm} \label{thm:k-partite}
For integers $n,m \ge 1$ and $k = 3,4$ with $km |n$,
\begin{align*}
	t^k_1(n,K^k_k(m)) = t^k_1(n,k)+o(n^{k-1}) = 
\begin{cases}
\left(\frac59 +o(1)\right) \binom{n}2 & \textrm{if $k=3$}\\
\left( \frac{37}{64}+o(1)\right) \binom{n}3 & \textrm{if $k=4$}.
\end{cases}
\end{align*}
\end{thm}
For $m=1$, H\'{a}n, Person and Schacht~\cite{MR2496914} evaluated $t^3_1(n,3)$ asymptotically.
K\"{u}hn, Osthus and Treglown~\cite{kuhn2010matchings} and independently Khan~\cite{2011arXiv1101.5830K} determined the exact value of $t^3_1(n,3)$ for large~$n$.
Also, Khan~\cite{2011arXiv1101.5675K} evaluated $t^4_1(n,4)$ exactly for large~$n$.

Recall that $t^k(n,t) = t^k_{k-1}(n,K^k_t)$.
For $t = k+1$, we give bounds on $t^k(n,k+1)$ for $k \ge 4$.
\begin{thm} \label{thm:t^k(n,k+1)}
Given an integer $k \ge 4$ and a constant $\gamma >0$, there exists an integer $n_0 = n_0(k, \gamma)$ such that for all $n \ge n_0$ with $(k+1)|n$
\begin{align*}
	t^k(n,k+1) \le 
	\left(1- \frac{k+\mathbf{1}_{k,\textrm{odd}}}{2k^2}+ \gamma \right)n,
\end{align*}
where $\mathbf{1}_{k,\textrm{odd}}=1$ if $k$ is odd and $\mathbf{1}_{k,\textrm{odd}}=0$ otherwise.
Moreover, $t^k(n,k+1) \ge 2n/3$ for even integers $k\ge 4$.
\end{thm}

For a general $t >k \ge 3$, we bound $t^k(n, t)$ from above.
First, we define the function $\beta(k,t)$ for $3 \le k < t$.
For $k = 3 <t$, we define
\begin{align}
\beta(3,t) = \frac{2}{t^2-3t+4}. \nonumber
\end{align}
Note that $\beta(3,4) = 1/4$.
Given integers $k,t, l$ with $4 \le k <t$ and $(k-2)/2 \le l \le (t-2)/2$, define
\begin{align}
\beta(k,t,l) & =  \min \left\{ \frac{2}{\binom{t}{k-1}+ \binom{l+1}{k-1}+\binom{t-l-1}{k-1}} ,
 \frac{1}{\binom{t-1}{k-1} +\binom{l}{k-1}},
\frac{1}{2\binom{2l+1}{k-1}} \right\}, \nonumber\\
\beta(k,t) & = \max \{\beta(k,t,l) : (k-2)/2 \le l \le (t-2)/2\}, \nonumber
\end{align}
where $\binom{x}{y} = 0 $ if $x <y$.

\begin{thm} \label{thm:t^k(n,t)}
Given integers $t > k \ge 3$ and a constant $\gamma>0$, there exists an integer $n_0 = n_0(k, t,\gamma)$ such that for all $n \ge n_0$ with $t|n$
\begin{align}
	t^k(n, t) 
\le \left( 1- \beta(k,t) + \gamma \right)  n.
\nonumber
\end{align}
In particular, $t^3(n, 4) \le (3/4 + \gamma) n$.
\end{thm} 
Pikhurko~\cite{MR2438870} showed that $t^3(n, 4) \ge 3n/4 -2$, so $t^3(n, 4) =(3/4 +o(1))n$.
After the submission of the paper, we have learnt that Keevash and Mycroft~\cite{keevash2011geometric} determined the exact value of $t^3(n,4)$ for large $n$ using a different method.

For $t = k+1 \ge 4$, we have $\beta(k,k+1) = \beta(k,k+1, (k-1)/2)$ or $\beta(k,k+1) = \beta(k,k+1, (k-2)/2)$.
Moreover, $\beta(k,k+1) \le (k+\mathbf{1}_{k,\textrm{odd}})/2k^2$ unless $k \in \{3,4,6\}$.
Hence, Theorem~\ref{thm:t^k(n,k+1)} is stronger than Theorem~\ref{thm:t^k(n,t)} unless $k = 4$ or~$6$.
By evaluating $\beta(k,t)$ for $4 \le k <t-1$, Theorem~\ref{thm:t^k(n,t)} gives the following corollary.

\begin{cor} \label{cor:main}
Given an integer $t-1> k \ge 4$ and a constant $\gamma >0$, there exists an integer $n_0 = n_0(k,t, \gamma)$ such that, for all $n \ge n_0$ with $t|n$, the following statements hold:
\begin{itemize}
% \item[\rm (i)] If $t = k+1$, then
% \begin{align*}
% 	t^k(n,k+1) \le 
% \begin{cases}
% 	\left( 3/4 + \gamma  \right) n & \textrm{if $k =3$} \\
% 	\left( 4/5+\gamma  \right) n & \textrm{if $k = 4$} \\
% 	\left( 19/21+\gamma  \right) n & \textrm{if $k = 6$}\\
% 	\left(1 - \frac{k+\mathbf{1}_{k, \textrm{odd}}}{2k^2} + \gamma\right)n	& \textrm{otherwise}.
% \end{cases}
% \end{align*}
\item[\rm (i)] If $5 \le k+1  < t < 3k/2 - 1$, then 
\begin{align*}
t^k(n,t) \le \left( 1- 2 \binom{t}{k-1}^{-1} + \gamma\right) n .
\end{align*}
\item[\rm (ii)] If $k \ge 6$ and $2t \ge 3(k-1) + \sqrt{5k^2 -22k+25}$, then 
\begin{align*}
	t^k(n,t) \le \left( 1- \binom{t-1}{k-1}^{-1} + \gamma \right) n.
\end{align*}
\end{itemize}
\end{cor}

On the other hand, we know less about the lower bounds on $t^k(n,t)$.
For $t > k \ge 3$, we know that $t^k(n, t) \ge \left(1- (k-1)/t  \right)n$.
Indeed, this is true by considering the $k$-graph $H$ of order $n$ such that there exists a vertex subset $W$ of size $|W| =(t-k+1)n/t-1$ and every edge in $H$ meets~$W$.
For $k=3$, we are able to give a better lower bound on $t^3(n, t)$.

\begin{prp} \label{prp:lowerboundt^3(n,t)}
For every integer $t \ge 4$, there exists an integer $n_0 = n_0(t)$ such that for all $n \ge n_0$ with $t | n$, we have $t^3(n, t) \ge (1 -  193(t-1)^{-2} \log (t-1)  )  n$.
\end{prp}

Since Theorem~\ref{thm:t^k(n,t)} implies that $t^3(n, t) \le (1 -  2/(t^2 - 3t+4 )+ \gamma ) n$ for large~$n$, we ask the following question.
\begin{qns}
How does $t^k(n, t) $ behave for $t \ge k \ge 3$ as $n \rightarrow \infty$?
Is $t^3(n, t) \sim \left(1 -  C \log{t} /t^{2} \right) n$ or $t^3(n, t) \sim \left(1 -  C  /t^{2} \right) n$ for large $n$ and some absolute constant $C$ independent of $t$ and $n$?
\end{qns}

As an auxiliary result, we also prove the following Ramsey result.
\begin{thm} \label{thm:R(B_l,K_t)}
Let $c, \eps >0$ be constants.
Let $t\ge 3$ and $\lambda \ge 1$ be integers such that $(\lambda t)^{2 \eps} (\log \lambda t )^{1-\eps} \le c t^2$.
Then, there exist a constant $c_2 = c_2(\eps, c)$ such that $1/96 \le c_2$ and
\begin{align}
\frac{\lambda t^2}{96 \log(\lambda t)} \le R(B_{\lambda}, K^3_{t}) 
\le \frac{c_2\lambda t^2}{ \log(\lambda t)}, \nonumber
\end{align}
where $B_\lambda = K_3^3(1,1,\lambda+1)$.
\end{thm}

\section{Layout of the paper and preliminaries} \label{sec:prelim}

The paper is structured as follows.
First, we set up some basic notation in Section~\ref{subsec:notation}.
In Section~\ref{sec:close}, we study some properties of being $(F,i ,\eta)$-close.
Next, we look at the Ramsey number $R(B_{\lambda},K_t^3)$, Theorem~\ref{thm:R(B_l,K_t)}, in Section~\ref{subsec:Ramsey}.
In Section~\ref{sec:absorptionlemma}, we prove the absorption lemma, Lemma~\ref{lma:absorptionlemma}. 
Then, we apply the absorption lemma to prove Theorem~\ref{thm:k-partite}, that is evaluating $t^k_1(n,K_k^k(m))$ asymptotically for integers $k=3,4$ and $m \ge 1$ in Section~\ref{sec:K(m)factors}. 
The remainder of the paper is focused on bounding $t^k(n,t)$.
In Section~\ref{sec:construction}, we bound $t^k(n,t)$ from below.
We give a lower bound on $t^k(n,k+1)$ for $k \ge 4$ even, followed by a lower bound on $t^3(n,t)$, which proves Proposition~\ref{prp:lowerboundt^3(n,t)}.
In Section~\ref{sec:t^k(n,k+1)}, we bound $t^k(n,k+1)$ from above proving Theorem~\ref{thm:t^k(n,k+1)}.
With further work, we bound $t^k(n,t)$ from above, which proves Theorem~\ref{thm:t^k(n,t)} in Sections~\ref{sec:connect}.
Finally in Section~\ref{sec:remarks}, we prove Corollary~\ref{cor:main}.

\subsection{Notation} \label{subsec:notation}
For $a \in \mathbb{N}$, we refer to the set $\{1, \dots, a\}$ as $[a]$.

Throughout this paper, $H$ is assumed to be a $k$-graph $H$ of order~$n$.
The \emph{maximum $l$-degree} $\Delta_l(H)$ is simply the maximal $\deg (T)$ over all $l$-sets $T \subseteq V$.
Given an $l$-set $T \subseteq V$, the \emph{neighbourhood $N(T)$ of $T$} is the set of $(k-l)$-sets $S \subseteq V $ such that $T \cup S$ is an edge in~$H$.
Clearly, $\deg (T) = |N(T)|$.
Given an $s$-set $S \subseteq V$, define 
\begin{align*}
L(S) = 
\begin{cases}
V \setminus S & \textrm{if }s <k-1\\
\bigcap_{T \in \binom{S}{k-1}} N(T) & \textrm{if }s \ge k-1.
\end{cases}
\end{align*}
Note that if $S$ forms a $K_s^k$ in $H$, then $L(S)$ is precisely the set of vertices $v$ such that $S \cup v$ forms a $K_{s+1}^k$ in~$H$.

\subsection{$(F, i, \eta)$-close} \label{sec:close}

Let $H$ and $F$ be $k$-graphs on $n$ and $t$ vertices respectively.
Given an integer $i \ge 1$ and vertices $x,y \in V(H)$, we say that the vertex set $S \subseteq V(H)$ is an \emph{$(x,y)$-connector of length $i$ with respect to $F$} if $S \cap \{x,y\} = \emptyset$, $|S| = it -1$ and both $H[S \cup x]$ and $H[S \cup y]$ contain $F$-factors.
Therefore, $x$ and $y$ are $(F, i, \eta)$-close to each other if there exist at least $\eta n^{it-1}$ $(x,y)$-connectors of length~$i$ with respect to~$F$.
Given a vertex $x \in V(H)$, we denote by $\widetilde{N}_{F,i,\eta} (x)$ the set of vertices that are $(F,i, \eta)$-close to~$x$.
A subset $U \subseteq V$ is said to be \emph{$(F,i , \eta)$-closed in $H$} if every vertex in $U$ is $(F,i, \eta)$-close to all other vertices in~$U$.
This implies that $H$ is $(F,i , \eta)$-closed if $V(H)$ is $(F,i , \eta)$-closed in~$H$.
Given vertex sets $X,Y \subseteq V$, a triple $(x,y,S)$ is an \emph{$(X,Y)$-bridge of length~$i$ with respect to $F$} if $x \in X$, $y \in Y$ and $S$ is an $(x,y)$-connector of length $i$ with respect to $F$.
If $u \in X \cap Y$, then we say $(u,u, \emptyset)$ is an $(X,Y)$-bridge of length~$0$.

Next we study some basic properties of $(F,i,\eta)$-closeness.

\begin{prp} \label{prp:i-closedadditive}
Let $F$ be a hypergraph on $t$ vertices.
Let $i \ge 1$ be an integer and let $\eta, \varepsilon >0$ be constants.
Then there exist an integer $n_0 = n_0(t,i,\eta, \eps)$ and a constant $\eta_0 = \eta_0(t,i,\eta, \eps)>0$ satisfying the following:
Suppose that $H$ is a hypergraph of order $n \ge n_0$ and there exists a vertex $x \in V(H)$ with $| \widetilde{N}_{F,i, \eta}(x)| \ge \varepsilon n$.
Then, for all $0< \eta' \le \eta_0$, $\widetilde{N}_{F,i, \eta}(x) \subseteq  \widetilde{N}_{F,i+1, \eta'}(x)$.
\end{prp}

\begin{proof}
We may assume that $n_0$ is chosen to be sufficiently large.
Let $y \in \widetilde{N}_{F,i, \eta}(x)$ and $m = it-1$.
If $y$ is $(F,i+1, \eta')$-close to~$x$, then $y$ is $(F,i+1, \eta'')$-close to~$x$ for all $0 < \eta'' \le \eta'$.
Therefore, to prove the proposition, it is enough to show that $y$ is $(i+1,\eta_0)$-close to $x$ for some constant $\eta_0 >0$ depending only on $t,i,\eta, \eps$.
Since $y \in \widetilde{N}_{F,i, \eta}(x)$, there are at least $\eta n^m$ $(x,y)$-connectors $S$ of length~$i$ (with respect to~$F$).
Pick an $(x,y)$-connector $S$ of length~$i$.
Let $z \in \widetilde{N}_{F,i, \eta}(x) \setminus ( S \cup \{x,y\} )$.
There are at least $\eta n^m$ $(x,z)$-connectors $S'$ of length~$i$.
Moreover, the number of $S'$ containing a vertex in $S \cup y$ is at most $(m+1) n^{m-1} < \eta n^m/2$.
Hence, there are at least $\eta n^m/2$ $(x,z)$-connectors $S'$ (of length~$i$) with $S' \cap (S \cup y) = \emptyset$.
Since $H[S' \cup z]$ contains an $F$-factor, there is a $t$-set $T$ such that $z \in T \subseteq S' \cup z$ and $H[T]$ contains a copy of $F$.
By an averaging argument, the number of $t$-sets $T$, such that $T \cap (S \cup \{x, y\}) = \emptyset$ and $H[T]$ contains a copy of $F$, is at least 
\begin{align*}
\frac{\eta n^m/2}{n^{m-t+1}} \cdot \frac{\varepsilon n - m-2}{t} > \frac{\eta \varepsilon n^t}{4t}.
\end{align*}
Recall that $S$ is an $(x,y)$-connector of length $i$.
So $S \cup T$ is an $(x,y)$-connector of length $i+1$.
Note also that there are 
\begin{align*}
	\frac{ \eta \varepsilon n^t/(4t) \cdot \eta n^m/2}{\binom{m+t}{t}} = \frac{\eta^2 \eps}{8t \binom{m+t}t} n^{m+t}
\end{align*}
such choices for $S \cup T$.
Set $\eta_0 = \eta^2 \eps/(8t \binom{m+t}t)$.
Hence, $y$ is $(F,i+1,\eta_0)$-close to~$x$.
The proof of the proposition is completed.
\end{proof}

Let $x,y$ be distinct vertices in $V(H)$.
Let $X = \widetilde{N}_{F,i_X,\eta_X}(x)$ and $Y = \widetilde{N}_{F,i_Y,\eta_Y}(y)$.
In the next lemma, we show how $(X,Y)$-bridges are used to show that $x$ and $y$ are $(F, i, \eta)$-close to each other.

\begin{lma} \label{lma:bridge}
Let $F$ be a hypergraph on $t$ vertices.
Let $i_X,i_Y >0$ and $i \ge 0$ be integers and let $\eta_X, \eta_Y, \eta, \varepsilon >0$ be constants.
Then there exist an integer $n_0 = n_0(t,i,i_X,i_Y,\eta,\eta_X, \eta_Y, \eps)$ and a constant $\eta_0 = \eta_0(t,i,i_X,i_Y,\eta,\eta_X, \eta_Y, \eps)$ satisfying the following:
Suppose that $H$ is a hypergraph of order $n \ge n_0$.
Suppose that for distinct $x,y \in V(H)$, there are at least $\varepsilon n^{it+1}$ copies of $(X,Y)$-bridges of length~$i$ with respect to~$F$, where $X = \widetilde{N}_{F,i_X,\eta_X}(x)$ and $Y = \widetilde{N}_{F,i_Y,\eta_Y}(y)$.
Then, $x$ and $y$ are $(F,i_X+i_Y+i,\eta')$-close to each other for all $0< \eta' \le \eta_0$.
In particular, if $|X \cap Y| \ge \eps n$ (i.e. $i =0$), then $x$ and $y$ are $(F,i_X+i_Y,\eta')$-close to each other for all $0< \eta' \le \eta_0$.

Furthermore, if $X$ is $(F,i_X,\eta_X)$-closed and $Y$ is $(F,i_Y,\eta_Y)$-closed in $H$ and $|X|,|Y| \ge \varepsilon n$, then $X \cup Y$ is $(i_X+i_Y+i,\eta')$-closed in~$H$ for all $0< \eta' \le \eta_0$.
\end{lma}

\begin{proof}
We may assume that $n_0$ is chosen to be sufficiently large.
Let $i_0 = i_X+i_Y+i$, $m_0 = i_0t-1$, $m = it-1$, $m_X = it_X-1$ and $m_Y = it_Y-1$.
The number of $(X,Y)$-bridges $(x',y',S)$ of length $i$ (with respect to $F$) with $\{x,y\} \cap (S \cup \{x',y'\}) \ne \emptyset$ is at most $(m+2)n^{m+1} < \varepsilon n^{m+2}/2$.
Hence, the number of $(X,Y)$-bridges $(x',y',S)$ with $x' \in X \setminus (S \cup \{ x,y\})$ and $y' \in Y \setminus (S \cup \{ x,y\})$ is at least $\varepsilon n^{m+2}/2$.
Fix one such $(X,Y)$-bridge $(x',y',S)$.
Since $x' \in X \setminus x$, the number of $(x,x')$-connectors $S_X$ of length $i_X$ such that $S_{X} \cap (S \cup \{ x,x',y,y'\}) = \emptyset$ is at least 
\begin{align*}
\eta_X n^{m_X} - (m+4)   n^{m_X-1} \ge \eta_X n^{m_X}/2.
\end{align*}
Pick one such $S_{X}$.
Similarly, the number of $(y,y')$-connectors $S_Y$ of length $i_Y$ such that $S_{Y} \cap (S \cup S_X \cup \{ x,x',y,y'\}) = \emptyset$ is at least 
\begin{align*}
\eta_Y n^{m_Y} - (m_X+m+4)n^{m_Y-1} \ge \eta_{Y} n^{m_Y}/2.
\end{align*}
Pick one such $S_{Y}$.
Set $S_0 = S_{X} \cup S_{Y} \cup S \cup \{x',y'\}$.
Note that $S_0$ is an $(x,y)$-connector of length $i_0$.
Moreover, there are at least 
\begin{align}
	\frac{1}{\binom{m_0}{m,1,1,m_X,m_Y}} \times \frac{\varepsilon n^{m+2}}{2} \times \frac{\eta_{X} n^{m_X}}2  \times \frac{\eta_{Y} n^{m_Y}}2 = \eta_0' n^{m_0} \nonumber
\end{align}
distinct $S_0$, where $\eta_0' = \binom{m_0}{m,1,1,m_X,m_Y}^{-1} \eps \eta_X \eta_Y>0$.
So $x$ and $y$ are $(F,i_0,\eta')$-close to each other for all $0 < \eta' \le \eta_0'$.
By Proposition~\ref{prp:i-closedadditive}, choose $0 < \eta_0 \le \eta'_0$ sufficiently small such that the following two statements also holds:
If $X$ is $i_X$-closed and $|X| \ge \eps n$ then $X$ is $(F, i_0, \eta')$-closed in $H$ for all $0 < \eta' \le \eta_0$.
If $Y$ is $i_Y$-closed and $|Y| \ge \eps n$ then $Y$ is $(F, i_0, \eta')$-closed in $H$ for all $0 < \eta' \le \eta_0$.
Therefore, the lemma follows.
\end{proof}

%%%%%%%%%%%%%%%%%%%%%%%%%%%%%%%%%%%%%%%%%%

\subsection{Ramsey number of 3-graphs} \label{subsec:Ramsey}

The \emph{Ramsey number} $R(S,T)$ of $k$-graphs $S$ and $T$ is the minimum integer $N$ such that if we edge-colour $K_N^k$ with colours red and blue then there exists a red monochromatic copy of~$S$ or a blue monochromatic copy of~$T$.
Given an integer $\lambda \ge 0$, let $B_{\lambda}$ be the 3-graph on vertex set $\{x,y,z_1, \dots, z_{\lambda+1} \}$ with edges $xyz_i$ for $1 \le i \le \lambda+1$.
In other words, $B_{\lambda} = K_3^3(1,1, \lambda+1)$.

First we bound $R(B_{\lambda}, K^3_{t})$ from below.
A \emph{partial $t$--$(n,k,\lambda)$ design} is a family $\mathcal{J}$ of $k$-sets in~$[n]$ such that every $t$-set~$T$ is contained in at most $\lambda$ $k$-sets in~$\mathcal{J}$.
Note that a partial $2$--$(n,3,\lambda)$ design does not contain a~$B_{\lambda}$.
We are going to construct a partial $t$--$(n,k,\lambda)$ design with a small independence number by modifying a construction of Kostochka, Mubayi, R{\"o}dl and Tetali~\cite{MR1848784}.
It should be noted that Grable, Phelps and R{\"o}dl~\cite{MR1337352} constructed $2$--$(n,k,\lambda)$ designs with small independence number when $n$ is an even power of a sufficient large prime.
Given integers $k \ge t$, we write $(k)_t$ to mean $k!/(k-t)!$.

\begin{prp} \label{prp:R(B_l,K_t)below} 
Let $k$, $t$, $\lambda$ and $x$ be positive integers with $x > k \ge 2t-1$, and $8(k)_t \binom{k-1}t \binom{k} t  \log (\lambda x) \ge (2 e)^{k-t}$.
Then there exists a partial $t$--$(n,k, \lambda)$ design $H$ with 
\begin{align}
n = \left\lceil \left( \frac{\lambda x^{k-1}}{ 8(k)_t\binom{k-1}t \binom{k}t  \log (\lambda x)}\right)^{1/(k-t)} \right\rceil
\label{eqn:nR(B_l,K_t)}
\end{align}
and the independence number of $H$ is less than~$x$.
\end{prp}

\begin{proof}
We consider the following constrained random process.
First we order all $k$-sets of $[n]$ at random: $E_1$, \dots, $E_{\binom{n}k}$.
Let $H_0$ be the empty graph on vertex set~$[n]$.
For $1 \le j \le \binom{n}k$, set $H_j = H_{j-1} \cup E_j$ if $H_{j-1} \cup E_j$ is a partial $t$--$(n,k, \lambda)$ design, otherwise set $H_j = H_{j-1}$.
Let $H = H_{\binom{n}k}$.
Our aim is to show that with positive probability the independence number of $H$ is less than~$x$.

Fix an $x$-set~$X$.
Let $B_X$ be the event that $X$ is an independent set in~$H$.
Given a $k$-set $T \in \binom{X}{k}$, an edge $E$ in $H$ is called a \emph{$T$-witness} if $E$ precedes $T$ in the ordering and $|E \cap T| \ge t$.
Thus, in order for $B_X$ to happen, each $T\in \binom{X}{k}$ has at least $\lambda$ $T$-witnesses.
Each edge $E$ in $H$ with $E \not \subseteq X$ can be a $T$-witness for at most $\binom{k-1}{t} \binom{x-t}{k-t}$ $k$-sets $T \subseteq X$.
Therefore, if $B_X$ happens, then there are at least
\begin{align}
	m = \frac{\lambda \binom{x}k}{\binom{k-1}{t} \binom{x-t}{k-t}} \nonumber
\end{align}
witnesses, where an edge $E$ in $H$ is a \emph{witness} if $E$ is a $T$-witness for some $T\in \binom{X}{k}$.
Note that if $E$ is a witnesses, then $|X \cap E| \ge t$.

For $j \ge 1$, let $A_j = A_{X,j}$ denote the event that $H[X]$ is empty and there are at least $j$ edges $E$ in $H$ such that $t \le |E \cap X| < k$, i.e. there are at least $j$ witnesses in $H$.
Note that $B_X$ implies $A_m$. 
Our task is to bound the probability of $A_m$ from above by $\binom{n}{x}^{-1}$.
Let $E_{l_j}$ be the $j$th witness in the ordering.
Let $H^j$ be the $k$-graph $H_{l_j-1}$.
So $H^j$ has exactly $j-1$ witnesses while the next graph in the process has $j$ witnesses.
For $1 \le j \le m$, let $\mathcal{S}_j$ be the set of witnesses $E$ such that $H^{j} \cup E$ is a partial $t$-$(n,k,\lambda)$ design.
Hence, for all $S \in \mathcal{S}$ with $|X \cap S| \ge t$, each $U \in \binom{S}{t}$ is contained in fewer than $\lambda$ edges $E \in H^j$.
Let $\mathcal{T}_j$ be the set of $T \in \binom{X}{k}$ such that $T$ has fewer than $\lambda$ witnesses in $H^j$.
Clearly, $\mathcal{T}_j \subseteq \mathcal{S}_j\cap \binom{X}k$.
Recall that an edge $E$ with $E \not\subseteq X$ can be a $T$-witness for at most $\binom{k-1}{t} \binom{x-t}{k-t}$ $k$-sets $T \subseteq X$.
Consequently, for $1 \le j \le \lceil m/2 \rceil$, 
\begin{align}
	\left|\mathcal{S}_j\cap \binom{X}k\right|  & 
	\ge | \mathcal{T}_j|
> 	\binom{x}k - \frac{j-1}{\lambda}\binom{k-1}{t} \binom{x-t}{k-t} \nonumber \\
	& \ge \binom{x}k - \frac{\lceil m/2 \rceil -1}{\lambda}\binom{k-1}{t} \binom{x-t}{k-t} \ge \frac1{2} \binom{x}k. \label{eqn:SjXk}
\end{align}
Since $A_m \subseteq A_{m-1} \subseteq \dots \subseteq A_1$, we have
\begin{align*}
	\mathbb{P}(A_{m}) = \mathbb{P}(A_1)\mathbb{P}(A_2|A_1) \dots \mathbb{P}(A_m|A_{m-1}).
\end{align*}
Note that the events $A_1$ corresponds to a random choice from the set $\mathcal{S}_1$ with the result that the chosen set belongs to $\mathcal{S}_1 \setminus \binom{X}k$.
Similarly, for $j=2,\dots,m-1$, the event $A_{j}|A_{j-1}$ corresponds to a random choice from the set $\mathcal{S}_j$ with the result that the chosen set belongs to $\mathcal{S}_j \setminus \binom{X}k$.
Since $|\mathcal{S}_j| \le \binom{x}{t} \binom{n}{k-t}$ for all $1 \le j \le \lceil m/2 \rceil $, we have
\begin{align}
	\mathbb{P}(A_1) = \frac{|\mathcal{S}_1| - \binom{x}k}{|\mathcal{S}_1|} \le 1- \frac{\binom{x}k}{\binom{x}t \binom{n}{k-t}} < 1- \frac{\binom{x}k}{2\binom{x}t \binom{n}{k-t}}. \nonumber
\end{align}
Furthermore, for $1 < j \le \lceil m/2 \rceil $,
\begin{align}
	\mathbb{P}(A_j |A_{j-1}) = \frac{|\mathcal{S}_j \setminus \binom{X}k|}{|\mathcal{S}_j|} 
= 1- \frac{|\mathcal{S}_j\cap \binom{X}k|}{|\mathcal{S}_j|} 
\overset{\eqref{eqn:SjXk}}{\le} 1-\frac{\binom{x}k }{2\binom{x}t \binom{n}{k-t}}. \nonumber
\end{align}
This yields
\begin{align}
	\mathbb{P}(A_m) & \le  \mathbb{P}(A_1) \prod_{1< j \le \lceil m/2 \rceil} \mathbb{P}(A_j|A_{j-1})
	<  \left( 1 - \frac{ \binom{x}k }{2\binom{x}t \binom{n}{k-t}} \right)^{ m/2 }  
	\nonumber
\\
	& \le  \exp \left( - \frac{ m \binom{x}k }{4\binom{x}t \binom{n}{k-t}} \right)  
% 	=	\exp \left( - \frac{\lambda (x)_k}{4 (n)_{k-t} (k)_t \binom{k-1}{t} \binom{k}{t}}
% \right)\\
	 \le \exp \left( - \frac{\lambda x^k}{8 n^{k-t} (k)_t \binom{k-1}{t} \binom{k}{t}}\right)
 \nonumber \\
& \overset{\eqref{eqn:nR(B_l,K_t)}}{\le} \exp(- x \log (\lambda x)) = (\lambda x)^{-x}. \label{eqn:PAm}
\end{align}
On the other hand, by our assumptions on $k,t, \lambda$ and $x$, we have 
\begin{align}
	\frac{en}x & \overset{\eqref{eqn:nR(B_l,K_t)}}{\le} 2e \left( 8(k)_t\binom{k-1}t \binom{k}t  \log ( \lambda x ) \right)^{-1/(k-t)} ( \lambda x^{t-1})^{1/(k-t)} \nonumber \\
	& \le ( \lambda x^{t-1})^{1/(k-t)}
	\le \lambda x. \nonumber
 \end{align}
Since $\binom{n}{x} \le (en/x)^x \le (\lambda x)^x$, \eqref{eqn:PAm} implies that $\mathbb{P}(A_m) < \binom{n}{x}^{-1}$.
Thus the probability that $X$ is an independent set in $H$ is  less than $\binom{n}{x}^{-1}$.
Since $X$ is an $x$-subset of $[n]$ chosen arbitrarily, the union bound implies that with positive probability no $x$-subset of $[n]$ is an independent set.
Therefore, there exists an $H$ with $\alpha(H) < x$.
\end{proof}

\begin{proof}[Proof of Theorem~\ref{thm:R(B_l,K_t)}]
Note that $\log(\lambda t) \ge \log 3 \ge e/48$.
By Proposition~\ref{prp:R(B_l,K_t)below} (taking $k=3$, $t=2$ and $x = t$), there exists a $2$--$(n ,3,\lambda)$ design with independence number at most $t$ and $n  = \left\lceil \lambda t^2/(96 \log (\lambda t) ) \right\rceil$. 
Since a $2$--$(n ,3,\lambda)$ design does not contain a~$B_{\lambda}$, $R(B_{\lambda}, K^3_{t}) > \lambda t^2/(96 \log (\lambda t) )$.
 
To prove the upper bound, let $n = \lceil c_2 \lambda t^2 / \log(\lambda t)\rceil$ and $\tau = (\lambda n)^{1/2}$, where $c_2$ is a large constant independent of $n, \lambda, t$ to be chosen later.
Let $H$ be a 3-graph of order $n$ with $\Delta_2(H) \le \lambda$, so $H$ does not contain a $B_{\lambda}$.
We are going to show that the independence number of $H$ satisfies $\alpha(H) \ge t$.
Note that $\Delta_1(H) \le \lambda n = \tau^2$.
Recall that $(\lambda t)^{2\eps } (\log \lambda t)^{1- \eps} \le c t^2$.
Hence, the number of 2-cycles, that is the number of $B_1$ in $H$, is at most $\binom{\lambda}2 \binom{n}2 \le n \tau^{3- \eps}$.
Furthermore, $\tau \gg 3$ since $c_2$ is large.
Then, by a theorem of Duke, Lefmann and R\"{o}dl~\cite[Theorem~3]{MR1370956} (taking $k=3$, $t = \tau$ and $\gamma = \eps$), there exists a constant $c'' = c''(3,\eps)>0$ such that 
\begin{align*}
	\alpha ( H ) 
& \ge c'' \frac{n}{\tau} \sqrt{\log \tau} 
= c'' \sqrt{n \log (\lambda n) / 2\lambda}\\
& \ge c'' \sqrt{c_2 t^2 \left(  1 - \frac{\log \log \lambda t }{2\log \lambda t } \right)} 
\ge c'' t \sqrt{c_2/2} \ge t.
\end{align*}
where the last inequality holds provided $c_2 \ge 2 (c''(3, \eps))^{-2}$.
Therefore, the complement of $H$ contains a $K^3_t$.
\end{proof}

\section{Proof of the absorption lemma} \label{sec:absorptionlemma}
Here we prove the absorption lemma, Lemma~\ref{lma:absorptionlemma}, of which the proof is based on the method of H\'{a}n, Person and Schacht~\cite{MR2496914}.

\begin{proof}[Proof of Lemma~\ref{lma:absorptionlemma}]
Let $H$ be a hypergraph of order $n \ge n_0$ such that $H$ is $(F,i,\eta)$-closed.
Throughout the proof we may assume that $n_0$ is chosen to be sufficiently large.
Set $m_1 = it-1$ and $m= (t-1) (m_1 +1)$.
Furthermore, call a $m$-set $A \in \binom{V}m$ an \emph{absorbing $m$-set for a $t$-set $T \subseteq V(H)$} if $A \cap T= \emptyset$ and both $H[A]$ and $H[A \cup T]$ contain $F$-factors.
Denote by $\mathcal{L}(T)$ the set of all absorbing $m$-sets for $T$.
Next, we show that for every $t$-set $T$, there are many absorbing $m$-sets for~$T$.

\begin{clm} \label{clm:numberofabsorbingm-set}
For every $t$-set $T \in \binom{V}t$, $|\mathcal{L}(T)| \ge (\eta/2)^{t}\binom{n}{m}$.
\end{clm}

\begin{proof}
Let $T= \{v_1, \dots, v_t\}$ be fixed.
Since $v_1$ and $u$ are $(F,i,\eta)$-connected for $u \notin T$, there are at least $\eta n^{m_1}$ $m_1$-sets $S$ such that $H[S \cup v_1]$ contains an $F$-factor.
Hence, by an averaging argument there are at least $\eta n^{t-1}$ copies of $F$ containing~$v_1$.
Since $n_0$ was chosen large enough, there are at most $(t-1)n^{t-2} \le \eta n^{t-1}/2$ copies of $F$ containing $v_1$ and $v_j$ for some $2 \le j \le t$.
Thus, there are at least $\eta n^{t-1}/2$ copies of $F$ containing $v_1$ but none of $v_2, \dots, v_{t}$.
We fix one such copy of $F$ with $V(F) = \{ v_1, u_2, \dots, u_t\}$.
Set $S_1 = \{ u_2, \dots, u_t\}$ and $W_0 = T$.

For each $2\le j \le t$, we are going to find a $(u_j, v_j)$-connector $S_j$ of length $i$ (with respect to $F$) such that $T, S_1, S_2, \dots, S_t$ are pairwise disjoint. 
Suppose that we have already found $S_1, \dots, S_{j-1}$ for some $2 \le j \le t$. 
We construct $S_j$ as follows.
Let $W = T \cup S_1 \cup \dots \cup S_{j-1}$.
Since $v_j$ and $u_j$ are $(F,i,\eta)$-close to each other, the number of $(u_j, v_j)$-connectors $S_j$ of length~$i$ is at least $\eta n^{m_1}$.
Note that at most $|W| n^{m_1-2} \le (2t + (t-2) m_1)n^{m_1-1} \le \eta n^{m_1}/2$ of them contains a vertex in $W$.
Hence there are at least $\eta n^{m_1}/2$ choices for $S_j$.
Therefore, we have constructed $S_1, \dots, S_t$. 
Note that $A = \bigcup_{1 \le j \le t} S_j$ is an absorbing set for $T$.
Recall that there are at least $\eta n^{t-1}/2$ choices for $S_1$ and at least $\eta n^{m_1}/2$ choices for $S_j$ for each $2 \le j \le t$.
In total, we obtain $(\eta/2)^{t} n^{m}$ absorbing $m$-sets for $T$ with multiplicity at most $m!$, so the claim holds.
\end{proof}

Now, choose a family $\mathcal{F}$ of $m$-sets by selecting each of the $\binom{n}{m}$ possible $m$-sets independently with probability $p = (\eta/2)^{t} n / (8 m^2 \binom{n}{m})$.
Then, by Chernoff's bound (see e.g.~\cite{MR1885388}) with probability $1-o(1)$ as $n \rightarrow \infty$, the family $\mathcal{F}$ satisfies the following properties:
\begin{align}
|\mathcal{F}| \le (\eta/2)^{t} n / (4 m^2 ) \qquad \text{and} \qquad
|\mathcal{L}(T) \cap \mathcal{F}| & \ge (\eta/2)^{2t} n / ( 16  m^2) \label{eqn:|F|}
\end{align} 
for all $t$-sets $T$.
Furthermore, we can bound the expected number of pairs of $m$-sets that are intersecting from above by
\begin{align}
	\binom{n}{m} \times m \times \binom{n}{m-1} \times p^2 
= \left(\frac{\eta}{2}\right)^{2t} \frac{n^2}{ 64 (n-m+1) m^2}
\le \left(\frac{\eta}{2}\right)^{2t} \frac{n}{ 64 m^2}.
\nonumber
\end{align}
Thus, using Markov's inequality, we derive that with probability at least~$1/2$
\begin{align}
	\textrm{$\mathcal{F}$ contains at most $\left(\frac{\eta}{2}\right)^{2t} \frac{n}{ 32 m^2}$ intersecting pairs of $m$-sets.} \label{eqn:F}
\end{align}
Hence, with positive probability the family $\mathcal{F}$ has all properties stated in \eqref{eqn:|F|} and \eqref{eqn:F}.
For each intersecting pair in $\mathcal{F}$, we delete one of the $m$-sets.
Further remove any $m$-set in $\mathcal{F}$ that is not an absorbing $m$-set for $T$ for all $t$-sets $T \subseteq V$.
Call the resulting family $\mathcal{F}'$ and set $U = V(\mathcal{F}')$.
Clearly, $|U| = m|\mathcal{F}'| \le m |\mathcal{F}| \le (\eta/2)^{t} n / (4m)$ by~\eqref{eqn:|F|}.
Note that $\mathcal{F}'$ consists of pairwise disjoint $m$-sets.
Since every $m$-set in $\mathcal{F}'$ is an absorbing $m$-set for some $t$-set $T$, $H[U]$ has an $F$-factor and so $|U| \in t \mathbb{Z}$.
For all $t$-sets $T$, by~\eqref{eqn:|F|} we have
\begin{align}
|\mathcal{L}(T) \cap \mathcal{F}'| & \ge (\eta/2)^{2t} n/ ( 16 m^2 ) - (\eta/2)^{2t} n/ (32 m^2) = (\eta/2)^{2t} n / (32 m^2).
\end{align}
For a set $W \subseteq V \setminus U$ of size $|W| \le (\eta/2)^{2t}t n / (32 m^2)$ and $ |W| \in t \mathbb{Z}$, $W$ can be partitioned in to at most $(\eta/2)^{2t} n / (32 m^2) $ $t$-sets.
Each $t$-set can be successively absorbed using a different absorbing $m$-set, so $H[U \cup W]$ contains an $F$-factor. 
\end{proof}

\section{$K_k^k(m)$-factors} \label{sec:K(m)factors}

Our aim of this section is to prove Theorem~\ref{thm:k-partite}, which determines the asymptotic values of $t^3_1(n,K_3^3(m))$ and $t^4_1(n,K_4^4(m))$.
The theorem is trivially  true for $m=1$, so we may assume that $m \ge 2$ for the remainder of this section.

Let $H$ be a $k$-graph. 
Given distinct $x,y \in V(H)$ and a constant $\alpha>0$, a $(k-1)$-set $S \in N(x) \cap N(y)$ is said to be \emph{$\alpha$-good for $(x,y)$} if $\deg (S) \ge \alpha n$.
Otherwise, $S$ is \emph{$\alpha$-bad} for~$(x,y)$.
A pair of vertices $(x,y)$ is \emph{$\alpha$-good} if the number of $\alpha$-good sets for $(x,y)$ is at least $\alpha \binom{n}{k-1}$.
If $(x,y)$ is not $\alpha$-good, then it is \emph{$\alpha$-bad}.
We are going to show that if $(x,y)$ is $\alpha$-good, then $x$ and $y$ are $(K_k^k(m),1, \eta)$-close to each other, Lemma~\ref{lma:alpha-good}.
First, we need the following simple facts.

The \emph{Tur\'an number} of a $k$-graph $F$, $\ex(n,F)$, is the maximum number of edges in an $F$-free $k$-graph of order $n$.
For $k,m \ge 2$, Erd\H{o}s~\cite{MR0183654} showed that $\ex(n,K^k_k(m)) < n^{k-m^{1-k}}$ for large $n$.
Furthermore, if $H$ is a $k$-graph of order $n$ with $e(H) > \ex(n,K^k_k(m)) + \beta n^{k}$ for $\beta >0$ and large $n$, then we have the `supersaturation' phenomenon discovered by Erd{\H{o}}s and Simonovits~\cite{MR726456}.

\begin{prp}[Supersaturation] \label{prp:erdos}
For integers $k,m \ge 2$ and constant $\beta>0$, there exist a constant $c = c(k,m,\beta) >0$ and an integer $n_0 =n_0(k,m,\beta) >0$ such that for every $k$-graph $H$ on $n \ge n_0$ vertices with at least $\beta n^k$ edges, there are at least $cn^{km}$ copies of $K^k_k(m)$ in $H$.
\end{prp}

\begin{lma} \label{lma:alpha-good}
Let $k, m \ge 2$ be integers and $\alpha>0$.
There exist a constant $\eta_0 = \eta_0(k,m,\alpha)>0$ and an integer $n_0= n_0(k,m,\alpha)$ such that for every $k$-graph $H$ of order $n \ge n_0$, if $(x,y)$ is $\alpha$-good in $H$ for $x,y \in V(H)$ then $x$ and $y$ are
$(K^k_k(m),1,\eta)$-close to each other for all $0 < \eta \le \eta_0$.
\end{lma}

\begin{proof}
Let $H$ be a $k$-graph of order $n$ sufficiently large.
Suppose that $(x,y)$ is $\alpha$-good for $x,y \in V(H)$.
In order to show that $x$ and $y$ are $(K_k^k(m),1, \eta)$-close to each other, it is sufficient to show that there exist at least $\eta n^{k m - 1}$ copies of $K^k_k(m, \dots, m, m+1)$ containing $x$ and $y$ in the partition class of size $m+1$.

Let $\mathcal{S}$ be the set of $\alpha$-good $(k-1)$-sets for $(x,y)$.
Clearly, $|\mathcal{S}| \ge \alpha \binom{n}{k-1}$ and each $S \in \mathcal{S}$ satisfies $\deg(S) \ge \alpha n$ and $x,y \in N(S)$.
Hence there exist at least $\alpha^2 n/2 $ vertices $z \in V(H) \setminus \{x,y\}$ so that $z \in N(S)$ for at least $\alpha^2 \binom{n}{k-1}/2$ sets $S \in \mathcal{S}$.
Otherwise, we have 
\begin{align*}
	\alpha \binom{n}{k-1} (\alpha n -2)  
	& \le \sum_{S \in \mathcal{S}} ( \deg(S) -2 )
	= \sum_{v \in V(H) \setminus \{x,y\}} | N(v) \cap \mathcal{S}| \\
	& < \frac{\alpha^2 n}2 \binom{n}{k-1} + \left( n-2 - \frac{\alpha^2 n}2 \right) \frac{\alpha^2}2 \binom{n}{k-1} \\
	& = \alpha^2 \binom{n}{k-1} \left( n - 1 - \frac{\alpha^2 n}4 \right),
\end{align*}
a contradiction as $n$ is large.
Let $Z$ be the set of vertices $z \in V(H) \setminus \{x,y\}$ such that $|N(x) \cap N(y) \cap N(z)| \ge \alpha^2 \binom{n}{k-1}/2$.
Thus, $|Z| \ge \alpha^2 n/2 $.
For $z \in Z$, consider the $(k-1)$-graph $H_z$ on vertex set $V(H) \setminus \{x,y,z\}$ with edge set $E(H_z) = N(x) \cap N(y) \cap N(z)$.
Note that $|E(H_z)| \ge \alpha^2 \binom{n}{k-1}/2$, so there are at least $c n^{(k-1)m}$ copies of $K^{k-1}_{k-1}(m)$ in $H_z$ by Proposition~\ref{prp:erdos}, where $c$ is a constant depending only on $\alpha, k$ and~$m$.
Moreover, each copy of $K^{k-1}_{k-1}(m)$ in $H_z$ corresponds to a $K^{k}_k(m, \dots, m , 3)$ in $H$ each of which the partition class of size~3 is precisely $\{x,y,z\}$. 

For each $(k-1)m$-set $T \subseteq V(H)$, denote by $\deg'(T)$ the number of vertices $z \in V \setminus \{x,y\}$ such that $T \cup \{x,y,z\}$ forms a $K^{k}_k(m, \dots, m , 3)$ in~$H$ of which the partition class of size~3 is precisely $\{x,y,z\}$. 
We get
\begin{align}
\sum_{ T \in \binom{V}{(k-1)m}}  \deg'(T) \ge c n^{(k-1)m} |Z| \ge c \alpha^2 n^{(k-1)m+1}/2.\nonumber
\end{align}
Therefore, the number of copies of $K^{k}_k(m, \dots, m , m+1)$, of which the partition class of size $m+1$ contains both $x$ and $y$, is equal to 
\begin{align}
	\sum_{ T \in \binom{V}{(k-1)m}} \binom{\deg'(T)}{m-1} & \ge  \binom{n}{(k-1)m} \binom{\sum \deg'(T)/\binom{n}{(k-1)m} }{m-1} \nonumber \\
	& \ge\binom{n}{(k-1)m} \binom{c \alpha^2 n^{(k-1)m+1} /(2\binom{n}{(k-1)m}) }{m-1} \nonumber \\
& \ge \eta_0 n^{km-1}, \nonumber
\end{align}
for some constant $\eta_0 = \eta_0(k,m,\alpha)$, where the first inequality is due to Jensen.
Therefore, $x$ and $y$ are $(K_k^k(m),1, \eta_0)$-close to each other as required.
\end{proof}

Now, suppose there exists a vertex $x_0 \in V(H)$ such that $(x,y)$ is $\alpha$-bad if and only if $x_0 \in \{x,y\}$.
The following proposition shows that if $\delta_1(H) \ge (1/2 +\gamma)\binom{n}{k-1}$, then every vertex is contained in many $K^{k}_k(m)$.
Let $T$ be a copy of $K_k^k(m)$ containing~$x_0$.
Thereby, every pair $(x,y)$ of vertices in $H_1 = H \setminus T$ is $\alpha/2$-good and so $H_1$ is $(K_k^k(m),1,\eta)$-closed by Lemma~\ref{lma:alpha-good}.

\begin{prp} \label{prp:K^k(l)}
Let $k, m \ge 2$ be integers and $\gamma>0$.
There exist a constant $c = c(k,m,\gamma )>0$ and an integer $n_0= n_0(k,m,\gamma)$ such that for every $k$-graph $H$ of order $n$ with $ \delta_1 (H) \ge (1/2 +\gamma)\binom{n}{k-1}$ and for every vertex $x\in V$, there exist at least $c n^{km-1}$ copies of $K^{k}_k(m)$ containing~$x$.
\end{prp}

\begin{proof}
Fix $x \in V$.
For every $y \in V \setminus \{x\}$, $|N(x) \cap N(y)| \ge 2 \gamma \binom{n}{k-1}$.
Let $\mathcal{S}$ be the set of $(k-1)$-sets $S \in N(x)$ such that $\deg(S) \ge \gamma n$.
Note that $| \mathcal{S} | \ge \gamma \binom{n}{k-1}$.
Otherwise, we have 
\begin{align*}
	2\gamma (n-1) \binom{n}{k-1} 
	& \le \sum_{ y \in V \setminus \{x\}} |N(x) \cap N(y)| 
	= \sum_{S \in N(x)} (\deg ( S )-1)\\
	& < \gamma \binom{n}{k-1} (n-1) + (1- \gamma) \binom{n}{k-1} \cdot \gamma n\\
	& = (2n -1 - \gamma n )\gamma \binom{n}{k-1},
\end{align*}
a contradiction as $n$ is large.
By an averaging argument, there exists a vertex $y \in V(H) \setminus \{x\}$ with $|N(y) \cap \mathcal{S}| \ge \gamma^2 \binom{n}{k-1}$.
Note that $(x,y)$ is $\gamma^2$-good.
By Lemma~\ref{lma:alpha-good}, $x$ and $y$ are $(K^k_k(m),1,\eta)$-close to each other, where $\eta>0$ is a constant depending only on $k,m,\gamma$.
Hence, there exist at least $\eta n^{km-1}$ $(km-1)$-sets $T$ such that $H[T \cup x]$ is a copy of $K_k^k(m)$.
Therefore, the proposition follows by setting $c = \eta$.
\end{proof}

We will also need the following result from Khan~\cite{2011arXiv1101.5830K,2011arXiv1101.5675K}, which shows the existence of an almost $K_k^k(m)$-factor for $k =3,4$.
The result was not stated explicitly, but it is easily seen from his proofs (for the non extremal cases).

\begin{lma}[Khan~\cite{2011arXiv1101.5830K,2011arXiv1101.5675K}] \label{lma:Khan}
Let $k=3,4$ and $m\ge 2$ be integers.
Let $\gamma$ be a strictly positive constant.
Then there exists an integer $n_0 = n_0(k,m,\gamma)$ such that every $k$-graph $H$ of order~$n \ge n_0$ with 
\begin{align*}
\delta_{1}(H) \ge  
\begin{cases}
\left(\frac59 +\gamma\right) \binom{n}2 & \textrm{if $k=3$} \\
\left( \frac{37}{64}+\gamma\right) \binom{n}3 & \textrm{if $k=4$}.
\end{cases}
\end{align*}
contains a set $\mathcal{T}$ of vertex-disjoint copies of $K^{k}_{k}(m)$ in $H$ covering all but at most $\gamma n$ vertices. 
\end{lma}

We are now ready to prove Theorem~\ref{thm:k-partite}.

\begin{proof}[Proof of Theorem~\ref{thm:k-partite}]
Let $k = 3,4$.
Recall that $t^k_1(n,K^k_k(m)) \ge t^k_1(n,K^k_k)$, so it is enough to show the upper bound and assume that  and $m \ge 2$. 
Let $\gamma>0$ be a small constant.
Set $\alpha, \eta_1, \eta_2, \gamma_2>0$ be new small constants depending only on $k, m$ and $\gamma$, whose values will become clear later on.
Let $n_0$ be a sufficiently large integer.
Let $H$ be a $k$-graph of order $n \ge n_0$ with $km | n$ and 
\begin{align*}
\delta_{1}(H) \ge  
\begin{cases}
\left(\frac59 +\gamma\right) \binom{n}2 & \textrm{if $k=3$} \\
\left( \frac{37}{64}+\gamma\right) \binom{n}3 & \textrm{if $k=4$}.
\end{cases}
\end{align*}
We are going to show that $H$ contains a $K^k_k(m)$-factor.
We follow the algorithm for finding $F$-factors as stated in Section~\ref{sec:introduction}.

\smallskip
\noindent
\textbf{Step 1}
Let $\alpha$ be chosen sufficiently small such that $\alpha < \gamma$ and  $4 \alpha^2 / \gamma \le  c(k, m, \gamma)/2$ and $4 \alpha^2 km n /\gamma \le \min\{ \alpha/4, \gamma/8\} $, where $c( k, m,\gamma)$ is the constant given by Proposition~\ref{prp:K^k(l)}.
For distinct $x,y \in V(H)$, note that $|N(x) \cap N(y)| \ge 2 \gamma \binom{n}{k-1}$. 
If $(x,y)$ is $\alpha$-bad, then all but at most $\alpha \binom{n}{k-1}-1 $ of $S \in N(x) \cap N(y)$ satisfies $\deg(S) < \alpha n$.
Hence, for each $\alpha$-bad pair $(x,y)$, there are at least $(2 \gamma - \alpha) \binom{n}{k-1}$ $(k-1)$-sets $S \in N(x) \cap N(y)$ with $\deg(S) < \alpha n$.

Next, we claim that there are at most $\binom{n}{k-1}/2$ $(k-1)$-sets $S \in \binom{V}{k-1}$ with $\deg (S) < \alpha n$.
Otherwise,
\begin{align}
	n \left( \frac12 + \gamma \right)\binom{n}{k-1} \le \sum_{x \in V} \deg (x) & = \sum_{S \in \binom{V}{k-1}} \deg (S)
		< \frac{(1+\alpha)n}2 \binom{n}{k-1} \nonumber
\end{align}
a contradiction.
Recall that if $S$ is $\alpha$-bad for $(x,y)$, then $S \in N(x) \cap N(y)$ and $\deg(S) < \alpha n$.
Hence, each such $S$ is $\alpha$-bad for at most $\binom{\deg(S)}2 < \binom{ \alpha n}2$ pairs of vertices.
Therefore, the number of $\alpha$-bad pairs in~$H$ is at most
\begin{align}
	\frac{ \binom{\alpha n}2  \binom{n}{k-1}/2}{(2 \gamma -\alpha ) \binom{n}{k-1}} 
%\le \frac{\alpha^2}{2\gamma - \alpha} \binom{n}2 
< \frac{\alpha^2}{\gamma} \binom{n}2. \nonumber
\end{align}
Hence, we have at most $4 \alpha^2 n /\gamma $ vertices that are in at least $n/4$ $\alpha$-bad pairs.
By Proposition~\ref{prp:K^k(l)} and a greedy algorithm, there exists a collection $\mathcal{T}_1$ of vertex-disjoint copies of~$K^k_k(m)$ covering these vertices (as $n$ is large) with $|V(\mathcal{T}_1)|\le 4 \alpha^2 km n /\gamma \le \min\{\alpha/4,\gamma/8\}n $.

Let $H_1 = H [V \setminus V(\mathcal{T}_1)]$ and $n_1 = |V(H_1)|$.
Note that for $x,y \in V(H_1)$ if $(x,y)$ is $\alpha$-good in~$H$, then $(x,y)$ is $(\alpha/2)$-good in~$H_1$.
Moreover, each vertex in $H_1$ is in at least $2n/3$ $(\alpha/2)$-good pairs. 
By Lemma~\ref{lma:alpha-good}, every vertex $x$ is $(K_k^k(m),1, \eta_1)$-close to at least $2n/3$ vertices in~$H_1$.
(Here $\eta_1$ is chosen to be sufficiently small such that $0<\eta_1 \le \eta_0(k,m,\alpha/2)$, where $\eta_0$ is the function given in Lemma~\ref{lma:alpha-good}.)
Hence, $|\widetilde{N}^{H_1}_{K^k_k(m),1 , \eta_1}(x) \cap \widetilde{N}^{H_1}_{K^k_k(m),1 , \eta_1}(y)| \ge n/3 > |n_1|/3$ for all $x,y \in V(H_1)$.
By Lemma~\ref{lma:bridge}, $H_1$ is $(K_k^k(m),2, \eta_2)$-closed provided that $\eta_2$ is sufficiently small.

\smallskip
\noindent
\textbf{Step 2} 
Set 
\begin{align*}
\gamma_2 = \frac{ (\eta_2/2)^{2km}} {2^7 k m(km-1)^2}
\end{align*}
We further assume that $\eta_2$ is small enough so that 
\begin{align*}
\frac{\left( \eta_2/{2} \right) ^{km}}{(4km(km-1))^2} \le \gamma /8
\quad \text{ and } \quad
\gamma_2 \le \gamma/2.
\end{align*}
Since $H_1$ is $(K_k^k(m),2, \eta_2)$-closed, there exists a vertex set $U$ in $H_1$ satisfying the conditions of the absorption lemma, Lemma~\ref{lma:absorptionlemma}.
Hence $|U| \le
% (\eta_2/2)^{km} n_1 / (4km(km-1))^2 \le
\gamma n_1/8$ and $|U| \in km \mathbb{Z}$.
Moreover, there exists an $F$-factor in $H[U \cup W]$ for every vertex set $W \subseteq V(H_1) \setminus U$ of size $|W|  \le \gamma_2 n_1 $ with $ |W| \in km \mathbb{Z}$.
Set $H_2 = H[V(H_1) \setminus U]$ and $n_2 = |V(H_2)|$.

\smallskip
\noindent
\textbf{Step 3}
Recall that $|V(\mathcal{T}_1)|,|U| \le \gamma n/8$ and $\gamma_2 \le \gamma/2$.
Hence
\begin{align*}
\delta_1(H_2) \ge \delta_1(H) - (|V(\mathcal{T}_1)| + |U|) \binom{n}{k-2}
\ge \begin{cases}
\left(\frac59 + \gamma_2 \right) \binom{n_2}2 & \textrm{if $k=3$} \\
\left( \frac{37}{64}+ \gamma_2 \right) \binom{n_2}3 & \textrm{if $k=4$}.
\end{cases}
\end{align*}
Apply Lemma~\ref{lma:Khan} on $H_2$ with $\gamma = \gamma_2$ and obtain a set $\mathcal{T}_2$ of vertex-disjoint copies of $K_k^k(m)$ such that $|V(H_2) \setminus V(\mathcal{T}_2)|  \le \gamma_2 n_2$.

\smallskip
\noindent
\textbf{Step 4}
Set $W = V(H_2) \setminus V(\mathcal{T}_2) = V(H) \setminus ( V(\mathcal{T}_1) \cup V(\mathcal{T}_2) \cup U)$.
Recall that $n, |V(\mathcal{T}_1)|, |V(\mathcal{T}_2)|,|U| \in km \mathbb{Z}$.
So $|W| \in km \mathbb{Z}$.
Note that $H_1[U \cup W]$ contains a $K_k^k(m)$-factor $\mathcal{T}_3$ by the choice of $U$ and $W$.
Thus, $\mathcal{T}_1 \cup \mathcal{T}_2 \cup \mathcal{T}_3$ is a $K_k^k(m)$-factor in~$H$.
\end{proof}

\begin{rmk}
We believe that one can determine the exact values of $t_1^k(n,K_k^k(m))$ for $k=3,4$ and $m \ge 2$, by using a stronger version of Lemma~\ref{lma:Khan}.
\end{rmk}

\section{Some lower bounds on $t^k(n,t)$} \label{sec:construction}

In this section, we give some lower bounds on $t^k(n,t)$ by constructing suitable families of $k$-graphs. 
First, we show that $t^k(n,k+1) \ge 2n/3$ for $k \ge 4$ even.

\begin{prp} \label{prp:t^k(n,k+1)up}
For even integers $k\ge 4$, $t^k(n,k+1) \ge 2n/3$.
\end{prp}

\begin{proof}
We define a $k$-graph $H$ on $n$ vertices as follows.
Partition $V(H)$ into three sets $V_1$, $V_2$ and $V_3$ of roughly the same size such that $|V_1| \ne |V_2| \pmod{2}$.
A $k$-set $S$ is an edge of $H$ if $|S \cap V_i|$ is odd for some $1 \le i \le 3$.
Observe that $\delta_{k-1}(H) \ge 2n/3-1$.
We now claim that $H$ does not contain a $K_{k+1}^k$-factor.
Let $T$ be a $K_{k+1}^k$ in~$H$.
If $|T \cap V_1|$ is even say, then we may assume without loss of generality that $|T \cap V_2|$ is even and $|T \cap V_3|$ is odd as $|T|$ is odd.
However, $T \setminus v$ is not edge in~$H$ for every $v \in T \cap V_3$, a contradiction.
Therefore, $|T \cap V_i|$ is odd for all $1 \le i \le 3$.
Recall that $|V_1| \ne |V_2| \pmod{2}$, so $H$ does not contain a $K_{k+1}^k$-factor.
\end{proof}

Our next task is to prove Proposition~\ref{prp:lowerboundt^3(n,t)}, that is, to show that $t^3(n,t) \ge (1- 193(t-1)^{-2} \log (t-1))n$ for some constant~$C$ independent of $n$ and $t$, where we generalise a construction given in Proposition~1 of Pikhurko~\cite{MR2438870}.

\begin{proof}[Proof of Proposition~\ref{prp:lowerboundt^3(n,t)}]
Let $l = R(B_{t}, K^3_{t-1}) -1$.
Note that $t-1 \ge \log (t-1)^2$. 
By Theorem~\ref{thm:R(B_l,K_t)} taking $\lambda=t-1$, $c=1$ and $\eps = 1/3$, we have $l  = R(B_{t-1}, K^3_{t-1}) -1 \ge (t-1)^2/ (192 \log (t-1))$ for some absolute constant~$c_0>0$ independent of~$t$.
Let $H_0$ be the 3-graph on vertex set $[l]$ with maximum codegree $\Delta_2(H_0) < t$ and independent number $\alpha(H_0) < t-1$, which exists by the choice of~$l$.

Let $n_0$ be a sufficiently large integer and let $n \ge n_0$ with $t |n$.
Partition $[n]$ into $A_0$, $A_1$, \dots, $A_l$ of size $a_0, a_1, \dots, a_l$ respectively such that $a_0+a_1+\dots+a_l = n$, $a_0$ is odd and $a_0/(t-1), a_1, \dots, a_l$ are nearly equal, that is, $|a_0/(t-1) - a_i|, |a_i - a_j| \le 2$ for all $1 \le i , j \le l$.
Let $H$ be a 3-graph on vertex set $[n]$ with edges satisfying one of following (mutually exclusive) properties:
\begin{enumerate}
	\item[(a)] lie inside $A_0$,
	\item[(b)] have two vertices inside $A_i$ and one in $A_0$ for $1 \le i \le l$,
	\item[(c)] have one vertex in each of $A_{i_1}$, $A_{i_2}$ and $A_{i_3}$ with $i_1i_2i_3 \in E(H_0)$.
\end{enumerate}
We claim that $\overline{H}$, the complement of $H$, does not contain a $K_t^3$-factor.
Let $T$ be a $K^3_{t}$ in $\overline{H}$, so $T$ is an independent set of size $t$ in~$H$.
By~(a), $|T \cap A_0| <3$.
If $|T \cap A_0| =1$, then without loss of generality $|T \cap A_i| = 1$ for $1 \le i \le t-1$ by~(b).
Together with $(c)$, we deduce that $[t-1]$ is an independent set in $H_0$ contradicting the fact that $\alpha( H_0) < t-1$.
Thus, every $K^3_{t}$ in $\overline{H}$ contains either 0 or 2 vertices in $A_0$.
So there is no $K_{t}^3$-factor in $\overline{H}$ as $|A_0| = a_0$ is odd.
Furthermore, 
\begin{align}
	\Delta_2(H) & \le \max_{I \in \binom{[l]}{t-1} }\{a_0, \sum_{i \in I}a_i  \} 
 =   \frac{(t-1) n }{(t-1) + l} + 2(t-1) \nonumber \\
& \le  \frac{192 n \log (t-1)}{192 \log (t-1) + (t-1)^2} + 2(t-1) 
\le  \frac{193 n \log (t-1)}{(t-1)^2} -2
\nonumber
\end{align}
as $n$ is large.
Since $\delta_2(\overline{H}) \ge n-2 - \Delta_2(H)$, the proposition follows. 
\end{proof}

\section{$K_{k+1}^{k}$-factors } \label{sec:t^k(n,k+1)}

Here, we prove Theorem~\ref{thm:t^k(n,k+1)}, which bounds $t^k(n,k+1)$ from above for $k \ge 4$.
We are going to show that if $H$ is a $k$-graph with $\delta_{k-1}(H) \ge \left(1- \frac{k+\mathbf{1}_{k,\textrm{odd}}}{2k^2} + \gamma \right)n$, then $H$ contains a $K^k_{k+1}$-factor.
The proof can be split into two main steps:
\begin{itemize}
\item[(a)] Showing $H$ is $(K^k_{k+1},1,\eta)$-closed.
\item[(b)] Finding an almost $K^k_{k+1}$-factor.
\end{itemize}
Note that (b) was solved in \cite[Corollary 1.7]{multihajnalsze}.

\begin{prp}[\cite{multihajnalsze}] \label{prp:approximate}
Let $3 \le k \le t$ be integers.
Then, given any $\varepsilon, \gamma >0$, there exists an integer $n_0 = n_0(k,t,\eps, \gamma)$ such that every $k$-graph $H$ of order $n > n_0$ with
\begin{align}
{\delta}_{k-1}(H) \ge \left( 1-\binom{t-1}{k-1}^{-1} + \gamma \right) n \nonumber
\end{align}
contains a $K_t^k$-matching $\mathcal{T}$ covering all but at most $\varepsilon n$ vertices. 
\end{prp}

Next, we are going to verify (a), that is, show that $H$ is $(K^k_{k+1},1,\eta)$-closed.

\begin{lma} \label{lma:K^k_k+1close}
Let $k \ge 3$ be an integer and $ \gamma >0$.
Then there exist a constant $\eta_0 = \eta_0(k,\gamma)$ and an integer $n_0 = n_0(k,\gamma)$ such that every $k$-graph $H$ of order~$n \ge n_0$ with 
\begin{align}
\delta_{k-1}(H) \ge \left(1- \frac{k+\mathbf{1}_{k,\textrm{odd}}}{2k^2}+ \gamma \right)n \nonumber,
\end{align}
is $(K^k_{k+1},1,\eta)$-closed for all $0 < \eta \le \eta_0$, where $\mathbf{1}_{k,\textrm{odd}}=1$ if $k$ is odd and $\mathbf{1}_{k,\textrm{odd}}=0$ otherwise.
\end{lma}

\begin{proof}
Let $x$ and $y$ be distinct vertices of~$H$.
Let $G$ be the $(k-1)$-graph on vertex set $V(H) \setminus \{x,y\}$ and edge set $E(G) = N(x) \cap N(y)$.
So 
\begin{align}
	\delta_{k-2} (G) \ge \left(1- \frac{k+\mathbf{1}_{k,\textrm{odd}}}{k^2} + 2\gamma \right)n > \left(1-\frac1{k-1} + \gamma\right) n. \label{eqn:K^k_k+1close}
\end{align}
For each edge $S$ of $G$, $|\bigcap_{v \in S} N^G(S \setminus v)| \ge (k-1)\gamma n$.
(Here, $N^G(T)$ is the neighbourhood of $T$ in $G$.)
Note that, for each $w \in \bigcap_{v \in S} N^G(S \setminus v)$, $S \cup w$ forms a $K_k^{k-1}$ in $G$.
Hence, the number of $K_k^{k-1}$ in~$G$ is at least
\begin{align*}
	\left(1-\frac1{k-1} + \gamma\right) \binom{n}{k-1} \frac{(k-1)\gamma n }{k} \ge (k-2) \gamma \binom{n}{k}.
\end{align*}

Let $T =\{v_1, \dots, v_k\} \subseteq V(G)$ be a $K_k^{k-1}$ in $G$.
For $u \in V(G) \setminus T$, we claim that if a vertex $u$ is in $N^G(S)$ for at least than $\lfloor k(k-2)/2\rfloor +1$ sets $S \in \binom{T}{k-2}$, then $(T \cup u) \setminus v_i$ forms a $K_{k}^{k-1}$ in $G$ for some $1 \le i \le k$.
Indeed, this is true by considering the 2-graph $G'_u$ on $T$ such that $v_{i_1}v_{i_2} \in E(G'_u)$ if and only if $u \in N^G(T \setminus \{v_{i_1},v_{i_2}\})$.
Note that $v_i$ has degree $k-1$ in $G'_u$ if and only if $(T \cup u) \setminus v_i$ forms a $K_{k}^{k-1}$ in $G$.
Therefore, if $u$ is in $N^G(S)$ for more than $\lfloor k(k-2)/2\rfloor $ $(k-2)$-sets $S \in \binom{T}{k-2}$, then $e(G'_u) > \lfloor k(k-2)/2\rfloor$.
So $G'_u$ contains a vertex $v_i$ of degree at least $k-1$ in~$G'_u$, i.e. $(T \cup u) \setminus v_i$ forms a $K_{k}^{k-1}$ as claimed.

Given a copy $T =\{v_1, \dots, v_k\}$ of $K_k^{k-1}$ in $G$ and an integer $1 \le i \le k$, let $U_i$ be the set of vertices $u \in V(G) \setminus T$ such that $(T \cup u)  \setminus v_{i}$ forms a $K_{k}^{k-1}$ in~$G$. 
By an averaging argument, there exists $i_0$ such that 
\begin{align*}
	|U_{i_0}| & \ge \frac1k \left(\binom{k}{k-2} \delta_{k-2}(G) - \left\lfloor \frac{k(k-2)}2\right\rfloor n   \right) \\
& \ge \left(\frac{k+\mathbf{1}_{k,\textrm{odd}}}{2k^2} + (k-1)\gamma\right)n
\end{align*}
by~\eqref{eqn:K^k_k+1close}.
Note that $|N^H(T \setminus v_{i_0}) \cap U_{i_0}| \ge \delta_{k-2}(G) + |U_{i_0}| \ge k \gamma n$.
Moreover, for each $z \in N^H(T \setminus v_{i_0}) \cap U_{i_0}$, both $(T \cup \{x,z\})  \setminus v_{i_0}$ and $(T \cup \{y,z\})  \setminus v_{i_0}$ form $K_{k+1}^k$ in~$H$.

Since there are $(k-2) \gamma \binom{n}{k}$ choices for $T$ and $k \gamma n $ choices for~$z$ (given~$T$ and $i_0$).
The number of $k$-sets $T'$ such that $H[x \cup T']$ and $H[y \cup T']$ are copies of $K_{k+1}^k$ is at least
\begin{align*}
	\frac{(k-2) \gamma \binom{n}{k}}{n} \cdot k \gamma n \cdot \frac1k \ge \frac{(k-2)\gamma^2}{2k!} n^k.
\end{align*}
Thus, $x$ is $(K^k_{k+1},1,\eta)$-closed to $y$, where $\eta_0 = (k-2)\gamma^2/(2k!)$.
Since $x$ and $y$ are arbitrary, $H$ is $(K^k_{k+1},1,\eta_0)$-closed as required.
\end{proof}

We are now ready to prove Theorem~\ref{thm:t^k(n,k+1)}.
Note that the second assertion of the theorem is implied by Proposition~\ref{prp:t^k(n,k+1)up}, so it is enough to prove the first assertion.

\begin{proof}[Proof of Theorem~\ref{thm:t^k(n,k+1)}]
Let $k \ge 4$ be an integer and $\gamma >0$.
Let $H$ be a $k$-graph of order $n \ge n_0$ with 
\begin{align}
\delta_{k-1}(H) \ge \left(1- \frac{k+\mathbf{1}_{k,\textrm{odd}}}{2k^2}+ \gamma \right)n
\nonumber
\end{align}
and $(k+1)|n$.
Throughout this proof, we may assume that $n_0$ is chosen to be sufficiently large.
Lemma~\ref{lma:K^k_k+1close} implies that there exists $\eta_0>0$ such that $H$ is $(K_{k+1}^k,1,\eta)$-closed for all $0<\eta \le \eta_0$.
Pick $0<\eta \le \eta_0$ such that $(\eta/2)^{k+1}/ ( 4(k+1)k)  \le \gamma/2$.
Let $U$ be the vertex set given by Lemma~\ref{lma:absorptionlemma} and so $|U| \le \gamma/2$.
Let $H' = H[V \setminus U]$.
Note that 
\begin{align}
	\delta_{k-1}(H') \ge \left(1- \frac{k+\mathbf{1}_{k,\textrm{odd}}}{2k^2}+ \frac{\gamma}2 \right)n \ge \left(1-\frac1{k+1}+\frac{\gamma}2 \right)n' \nonumber
\end{align}
where $n' = |V(H')| =  n - |U|$.
Let $\eps = (\eta/2)^{2(k+1} / (32 (k+1) k^2 )$.
Apply Proposition~\ref{prp:approximate} taking $\gamma = \gamma /2$ and obtain a family $\mathcal{T}$ of vertex-disjoint copies of $K^{k}_{k+1}$ in $H'$ covering all but at most $\eps n'$ vertices.
Let $W = V(H') \setminus V(\mathcal{T})$, so $|W| \le \eps n' \le \eps n$.
Recall that $n, |V(\mathcal{T})|, |U| \in (k+1)\mathbb{Z}$, so $|W| \in (k+1)\mathbb{Z}$.
By Lemma~\ref{lma:absorptionlemma}, there exists a $K_{k+1}^k$-factor $\mathcal{T}'$ in $H[U \cup W]$.
Thus, $\mathcal{T} \cup \mathcal{T}'$ is a $K_{k+1}^k$-factor in~$H$.
\end{proof}

%%%%%%%%%%%%%%%%%%%%%%%%%%%%%%%%%%%%%%%%%%%%%%%%%%%%%%

\section{$K^k_{t}$-factors for $t > k \ge 3$} \label{sec:connect}

Let $H$ be a $k$-graph of order~$n$ with $\delta_{k-1}(H) \ge \left(1- \beta(k,t) + \gamma \right)n $, where $\beta(k,t)$ is defined in Section~\ref{sec:introduction}.
Provided $n$ is sufficiently large, the next lemma shows that $H$ is actually $(K^k_{t},i, \eta)$-closed.

\begin{lma} \label{lma:(K^k_t,i,eta)close}
Let $3 \le k < t$ be integers and $0 < \gamma < \beta(k,t)$.
Then there exist integers $n_0 = n_0 (k,t, \gamma)$ and $i_0=i_0(k,t, \gamma)$ and a constant $\eta_0 = \eta_0(k,t, \gamma) >0$ such that every $k$-graph $H$ of order~$n \ge n_0$ with $\delta_{k-1}(H) \ge \left(1- \beta(k,t) + \gamma \right)n $ is $(K^k_{t},i_0,\eta)$-closed for all $0 < \eta \le \eta_0$.
\end{lma}

Observe that this lemma implies Theorem~\ref{thm:t^k(n,t)} by mimicking the proof of Theorem~\ref{thm:t^k(n,k+1)}, where we replace Lemma~\ref{lma:K^k_k+1close} with Lemma~\ref{lma:(K^k_t,i,eta)close}.
Thus, to prove Theorem~\ref{thm:t^k(n,t)}, it suffices to prove Lemma~\ref{lma:(K^k_t,i,eta)close}.
Its proof is divided into the following steps.
Let $H$ be a $k$-graph satisfying the lemma.
For $3 \le k < t$, define
\begin{align}
 d(k,t) & =  1 - \binom{t-1}{k-1} \beta(k,t). \label{eqn:d(k,t)}
\end{align}
Note that  $d(k,t) \ge 0$ for all $t > k \ge 3$ as $\beta(k,t)\le \binom{t-1}{k-1}^{-1}$.
Recall that $\widetilde{N}_{K_t^k,1,\eta_1}(v)$ is the set of vertices $u$ that are $(K_t^k,1,\eta_1)$-close to $v$.
First, we prove Proposition~\ref{prp:1connectionradius} showing that the size of $\widetilde{N}_{K_t^k,1,\eta_1}(v)$ is at least $(d(k,t)+ 2 \gamma)n$ for every vertex $v \in V(H)$.
Next, we partition $V(H)$ into bounded number of vertex classes $W_1, \dots,W_p$ such that $|W_j| \ge (d(k,t)+ \gamma)n$ and $W_j$ is $(K_t^k,i', \eta')$-closed in~$H$ for all $j \le p$, Lemma~\ref{lma:connectioncomponent}.
If there is only one vertex class, then Lemma~\ref{lma:(K^k_t,i,eta)close} holds.
So we may assume that $p \ge 2$.
Recall that for vertex sets $X,Y \subseteq V$, a triple $(x,y,S)$ is an $(X,Y)$-bridge of length $1$ with respect to $K_{t}^k$ if $x \in X$, $y \in Y$ and $S \in \binom{V(H)}{t-1}$ such that both $H[S \cup x]$ and $H[S \cup y]$ form copies of~$K_{t}^k$.
Using Lemma~\ref{lma:A_iB_j}, we show that there exist many say $(W_1,W_2)$-bridges.
By Lemma~\ref{lma:bridge}, we deduce that $W_1 \cup W_2$ is $(K_t^k,i'', \eta'')$-closed in~$H$.
Thus the number of vertex classes is reduced by one. 
Since $p$ is bounded, we can repeat this argument and merge all $W_i$'s into one class.

% We advise the reader to consider the special case when $k=3$ and $t=4$ which illustrate all the ideas in the proof.

\subsection{Proof of Lemma~\ref{lma:(K^k_t,i,eta)close}} \label{sec:pfclose}

Let $H$ be a $k$-graph of order~$n$ with $\delta_{k-1}(H) \ge \left(1- \beta(k,t) + \gamma \right)n $.
Given an $s$-set $S$ with $s \ge k-1$, recall that for $L(S) = \bigcap \{ N(T) : T \in \binom{S}{k-1} \}$.
Moreover, if $S$ forms a $K_s^k$ in $H$, then $L(S)$ is precisely the set of vertices~$v$ such that $S \cup v$ forms a $K_{s+1}^k$ in~$H$.
We obtain the following simple proposition, whose proof is immediate and omitted.
\begin{prp} \label{prp:L(S)}
Let $3 \le k < t$ be integers and $0 < \gamma < \beta(k,t)$.
Let $H$ be a $k$-graph of order~$n$ with $\delta_{k-1}(H) \ge \left(1- \beta(k,t) + \gamma \right)n $.
Then for every $s$-set $S \subseteq V(H)$ with $s \ge k-1$,
\begin{align}
	|L(S)| & \ge \left( 1 - \binom{s}{k-1} (\beta(k,t)-\gamma) \right)n. \nonumber
\end{align}
In particular, for every $(t-1)$-set $T \subseteq V(H)$, $|L(T)| \ge ( d(k,t) + 3 \gamma)n$.
\end{prp}

Next we show that $\widetilde{N}_{K^k_{t},1,\eta}(v)$ is large for every $v \in V(H)$.

\begin{prp} \label{prp:1connectionradius}
Let $3 \le k < t$ be integers and $0 < \gamma < \beta(k,t)$.
Then there exist an integer $n_0 = n_0 (k)$ and a constant $\eta = \eta(k,t,\gamma)>0$ such that every $k$-graph $H$ of order~$n \ge n_0$ with $\delta_{k-1}(H) \ge \left(1- \beta(k,t) + \gamma \right)n $ satisfies  $|\widetilde{N}_{K^k_{t},1,\eta}(v)| \ge (d(k,t)+2\gamma)n$ for all $v \in V(H)$.
\end{prp}

\begin{proof}
Write $\delta$, $\beta$ and $d$ for $\delta_{k-1}(H)$, $\beta(k,t)$ and $d(k,t)$ respectively.
Set
\begin{align}
 	\eta = \frac{\gamma }{2(t-1)!} \prod_{k-1 \le s \le t-1}  \left(1 - \binom{s}{k-1}(\beta - \gamma) \right) \nonumber.
\end{align}
Recall that $\beta \le \binom{t-1}{k-1}^{-1}$, so $\eta >0$.
By a greedy algorithm and Proposition~\ref{prp:L(S)}, each vertex $v$ is contained in at least
\begin{align}
	& \frac{1}{(t-1)!} (n-1) \dots (n-k+2) \prod_{k-1 \le s \le t-1} \left( \left(1 - \binom{s}{k-1}(\beta - \gamma)\right)n \right) \nonumber \\
	& \ge \frac{n^{t-1}}{2(t-1)!} \prod_{k-1 \le s \le t-1}  \left(1 - \binom{s}{k-1}(\beta - \gamma)\right) = \gamma^{-1} \eta n^{t-1} \nonumber
\end{align}
copies of $K^k_{t}$.
Fix $v \in V(H)$.
Let $W$ is the set of $(t-1)$-sets $T$ in~$V(H)$ such that $T \cup v$ forms a~$K^k_{t}$.
So $|W| \ge \gamma^{-1} \eta n^{t-1}$.
Define $G$ to be the bipartite graph with vertex classes $V(H)$ and $W$ such that, for all $x \in V(H)$ and all $T \in W$, $(x,T)$ is an edge in $G$ if and only if $ T \cup x$ forms a $K^k_{t}$ in~$H$.
Note that for all $T \in W$, we have $d^G(T) = |L(T)|  \ge \left(d+ 3 \gamma \right) n$ by Proposition~\ref{prp:L(S)}.
We now claim that there are more than $(d+2\gamma) n$ vertices $x \in V(H)$ with degree (in $G$) $d^G(x) \ge \gamma |W| \ge \eta n^{t-1}$.
Indeed this is true, since otherwise
\begin{align}
\left(d+ 3\gamma \right) n |W|
& \le e(G) = \sum_{x \in V(H)} d^G(x)  \nonumber \\
& \le \gamma n \cdot (|W|-(d+ 2\gamma)n ) + |W| (d+ 2\gamma)n, \nonumber
\end{align}
a contradiction.
The proposition follows as $x$ is $(K^k_{t},1, d^G(x)/n^{t-1})$-close to $v$ for all $x \in V(H) \setminus v$.
\end{proof}

Now, we show that we can partition the vertex set $V(H)$ into $W_1, \dots, W_p$ such that each $W_j$ is $(K^k_{t},i, \eta)$-closed in $H$.

\begin{lma} \label{lma:connectioncomponent}
Let $3 \le k < t$ be integers and $0 < \gamma< \beta(k,t)$.
Then there exist integers $n_1 = n_1(k,t, \gamma)$ and $i_1' = i_1'(k,t, \gamma)$ and a constant $\eta_1' = \eta_1'(k,t, \gamma)>0$ satisfying the following:
Let $H$ be a $k$-graph of order $n \ge n_1$ with $\delta_{k-1}(H) \ge \left(1- \beta(k,t) + \gamma \right)n $.
Then, there exists a vertex partition of $V(H)$ into $W_1, \dots, W_p$ such that, for each $1 \le j \le p$, $|W_j| \ge (d(k,t)+\gamma/2) n$ and $W_j$ is $(K_{t}^k, i_1', \eta)$-closed in $H$ for all $0 < \eta \le \eta_1'$. 
In particular, $p \le  (d(k,t)+\gamma/2)^{-1} $.
\end{lma}

Here we present an outline of the proof.
Proposition~\ref{prp:1connectionradius} implies that $|N_{K^k_t,1,\eta}(v)| \ge (d(k,t) + 2 \gamma)n$ for all $v \in V(H)$.
Suppose there exists a vertex set $U \subseteq V(H)$ that is $(K_t^k,i, \eta')$-closed in~$H$ with $|U| \ge (d + \gamma/2)n$.
If a vertex $y \in V(H) \setminus U$ satisfies $|N_{K_t^k,1, \eta}(y) \cap U| \ge \gamma n$, then $U \cup y$ is $(K_t^k,i+1, \eta'')$-closed by Lemma~\ref{lma:bridge}.
Now assume that $|N_{K_t^k,1, \eta}(y) \cap U| < \gamma n$ for all $ y \in V(H) \setminus U$.
Let $V' = V(H) \setminus U$.
Hence, $|N_{K^k_t,1,\eta}(v)| \ge (d(k,t) +  \gamma)|V'| $ for all $v \in V'$.
This suggests that there should also exist a vertex set $U' \subseteq V'$ that is $(K_t^k,i, \eta')$-closed in~$H$ with $|U'| \ge (d + \gamma/2)n$.
Assume that $U'$ does exist and set $V'' = V' \setminus U'$.
We repeat the arguments on $V''$ and find a $(K_t^k,i+1, \eta'')$-closed vertex set $U'' \subseteq V''$.
We iterate this process.
Since $U, U',U'',\dots$ are pairwise disjoint and each of size at least $(d + \gamma/2)n$, this process will terminate in a finite number of rounds.

\begin{proof}[Proof of Lemma~\ref{lma:connectioncomponent}]
Write $\delta$, $\beta$ and $d$ for $\delta_{k-1}(H)$, $\beta(k,t)$ and $d(k,t)$ respectively.
Let $\eta_0 = \eta(k,t, \gamma)$ be given by Proposition~\ref{prp:1connectionradius}.
Set $i_0 = \lceil - \log_2 d \rceil + \lceil 2^{(d(k,t)+ \gamma/2)^{-1}+2}/\gamma \rceil$ and $i'_1 = 2^{i_0}$.
Throughout this proof, $\eta_1, \eta_2, \dots, \eta_{i_0}$ is assumed to be a decreasing sequence of strictly positive small constants.
The precise value of each $\eta_i$ will become clear later on.
We would like to point out that we choose $\eta_1, \eta_2, \dots, \eta_{i_0}$ in turns before choosing~$n_0$. 
Hence, for all $1 \le i \le i_0$, $\eta_i$ depends only $k,t, \gamma, \eta_0, \eta_1, \dots, \eta_{i-1}$, i.e. $\eta_i = \eta_i(k,t, \gamma)$.
Now choose $n_0$ to be sufficiently large.

We write $i$-close to mean $(K^k_{t},2^i,\eta_i)$-close.
Given a vertex $v \in V(H)$, define $\widetilde{N}_i(v)$ to be the set of vertices $u$ that are $i$-close to $v$.
So Proposition~\ref{prp:1connectionradius} implies that $|\widetilde{N}_0(v)| \ge (d+2\gamma )n$.
Moreover, by our choices of $\eta_i$ and Proposition~\ref{prp:i-closedadditive}, we further assume that $\widetilde{N}_{i+1}(v) \subseteq \widetilde{N}_i(v) $ for all $0 \le i < i_0$ and all $v \in V(H)$ provided $|\widetilde{N}_i(v)| \ge \gamma/2$.
Hence, if $W \subseteq V(H)$ is $i$-closed in $H$ for some $i \le i_0$, then $W$ is $i_0$-closed provided $|W| \ge \gamma/2$.

For $0 \le r < (d(k,t)+ \gamma/2)^{-1}$, suppose that we have already found disjoint vertex sets $W_1, \dots W_r$ in $V(H)$ such that 
\begin{itemize}
	\item[(i)] $|W_j| \ge (d+\gamma/2) n  $ for all $1 \le j \le r$.
	\item[(ii)] $W_j$ is $i_0$-closed in $H$ for all $1 \le j \le r$.
	\item[(iii)] Let $V_r= V(H) \setminus \bigcup_{1 \le j \le r} W_j$.
If $V_r \ne \emptyset$, then 
\begin{align*}
|\widetilde{N}_{0}(v) \cap V_{r}| & \ge (d + (2 - \sum_{1 \le s \le r}2^{-s})\gamma)n &\text{for all $v \in V_r$.}
\end{align*}

\end{itemize}
If $V_r = \emptyset$, then the lemma holds.
Hence, we may assume that $V_r \ne \emptyset$.
To prove the lemma, it is enough to show that there exists a vertex set $W_{r+1}$ in $V_{r}$ so that $W_1, \dots W_{r+1}$ also satisfy (i)--(iii). 
We are going find $W_{r+1}$ in $V_{r}$ as follows.

% First suppose that there exists an integer $0 \le i < i_0$ such that $|\widetilde{N}_i(v) \cap V_r| \ge (1+\gamma)n/2$ for all $v \in V_r$.
% Note that $|\widetilde{N}_{i}(v) \cap \widetilde{N}_{i}(u)| \ge \gamma n$ for all $u, v \in V_r$.
% Thus, $V_r$ is $(i+1)$-closed in $H$ by Lemma~\ref{lma:bridge}.
% We are done by setting $W_{r+1} = V_r$.
% Therefore, we may assume that for all $0 \le i < i_0$ there exists a vertex~$v$ such that $|\widetilde{N}_i(v) \cap V_r| < (1+\gamma)n/2$.

Let $i'$ be the largest integer such that 
\begin{align}
	|\widetilde{N}_{i'}(v) \cap V_r| & \ge (2^{i'} d + \gamma)n &\text{ for all $v \in V_r$.} \label{eqn:Ni'}
\end{align}
Note that $i'$ exists and $0 \le i' \le \lceil - \log_2 d \rceil$, where the lower bound is due to~(iii).
Let $v_0$ be a vertex in $V_r$ such that $|\widetilde{N}_{i'+1}(v_0)| < (2^{i'+1} d + \gamma)n$.
Define $U$ to be the set of vertices $u \in V_r \cap \widetilde{N}_{i'}(v_0)$ such that 
\begin{align*}
|\widetilde{N}_{i'}(u) \cap  \widetilde{N}_{i'+1} (v_0)| & \ge (2^{i'}d +2\gamma/3)n.
\end{align*}

\begin{clm} \label{clm:U}
The size of $U$ is at least $(2^{i'}d + \gamma/2)n$ and $U$ is $(i'+1)$-closed in~$H$.
\end{clm}

\begin{proof}[Proof of claim]
For distinct $u, u' \in U$,  we have 
\begin{align}
|\widetilde{N}_{i'}(u) \cap  \widetilde{N}_{i'} (u')|  
&\ge 
|\widetilde{N}_{i'}(u) \cap \widetilde{N}_{i'+1}(v_0)| + |\widetilde{N}_{i'}(u') \cap \widetilde{N}_{i'+1}(v_0)|- |\widetilde{N}_{i'+1}(v)|  \nonumber \\
& \ge  \gamma n/3  \nonumber
\end{align}
as $|\widetilde{N}_{i'}(v_0)| < (2^{i'+1} d + \gamma)n$.
Hence, $U$ is $(i'+1)$-closed in $H$ by Lemma~\ref{lma:bridge}.
To complete the proof of the claim it is enough to show that $|U| \ge (2^{i'}d + \gamma/2)n$.
Note that if a vertex $w \in V \setminus v_0$ satisfies $|\widetilde{N}_{i'}(w) \cap  \widetilde{N}_{i'} (v_0)| \ge \gamma^2n/6$, then $w$ is $(i'+1)$-close to~$v_0$ by Lemma~\ref{lma:bridge}.
Thus, 
\begin{align}
|\widetilde{N}_{i'}(v_0) \cap \widetilde{N}_{i'} (w)| & < \gamma^2 n/6 \nonumber &\text{for all $w \in V(H) \setminus \widetilde{N}_{i'+1}(v_0)$.}
\end{align}
Let $G$ be the auxiliary bipartite graph with vertex classes $\widetilde{N}_{i'}(v_0)$ and $V(H) \setminus \widetilde{N}_{i'+1}(v_0)$ such that $u \in \widetilde{N}_{i'}(v_0)$ is joined to $w \in V(H) \setminus \widetilde{N}_{i'+1}(v_0)$ if and only if $w \in \widetilde{N}_{i'}(u) $.
Hence, for each $u \in \widetilde{N}_{i'}(v_0)$, $d^G(u) = |\widetilde{N}_{i'}(u) \setminus \widetilde{N}_{i'+1}(v_0)|$.
For each $w \in V(H) \setminus \widetilde{N}_{i'+1}(v_0)$, $d^G(w) = |\widetilde{N}_{i'}(v_0) \cap \widetilde{N}_{i'} (w)| < \gamma^2 n/6$.
Therefore, we have
\begin{align}
 \sum_{u \in \widetilde{N}_{i'}(v_0)} |\widetilde{N}_{i'}(u) \setminus \widetilde{N}_{i'+1}(v_0)| 
& = \sum_{u \in \widetilde{N}_{i'}(v_0)}  d^G(u) = \sum_{w  \notin \widetilde{N}_{i'+1}(v_0)} d^G(w) 
\nonumber
\\
& < \gamma^2 n^2 /6.
\label{eqn:sumw}
\end{align}
Recall \eqref{eqn:Ni'} and the definition of $U$.
For $u' \in (V_r \cap \widetilde{N}_{i'}(v_0)) \setminus U$,
\begin{align*}
|\widetilde{N}_{i'}(u') \setminus \widetilde{N}_{i'+1}(v_0) | =  |\widetilde{N}_{i'}(u')| - |\widetilde{N}_{i'}(u') \cap \widetilde{N}_{i'+1}(v_0)| > \gamma n/3.
\end{align*}
By summing $|\widetilde{N}_{i'}(u') \setminus \widetilde{N}_{i'+1}(v_0) |$ over all $u' \in (V_r \cap \widetilde{N}_{i'}(v_0)) \setminus U$ and \eqref{eqn:sumw}, we have
\begin{align*}
	\gamma n   |(V_r \cap \widetilde{N}_{i'}(v_0)) \setminus U|  /3
& \le \sum_{u' \in (V_r \cap \widetilde{N}_{i'}(v_0)) \setminus U} |\widetilde{N}_{i'}(u') \setminus \widetilde{N}_{i'+1}(v_0)| \\
& \le \sum_{u \in \widetilde{N}_{i'}(v_0)} |\widetilde{N}_{i'}(u) \setminus \widetilde{N}_{i'+1}(v_0)|
 < {\gamma^2 n^2}/6.
\end{align*}
Hence, $|(V_r \cap \widetilde{N}_{i'}(v_0)) \setminus U| \le \gamma n/2$.
Since $|\widetilde{N}_{i'}(v_0) \cap V_r| \ge (2^{i'} d + \gamma)n$ by~\eqref{eqn:Ni'}, we deduce that $|U| \ge (2^{i'} d + \gamma/2)n$ as desired.
This completes the proof of the claim.
\end{proof}

Set $U_0 = U$.
For an integer $j \ge 1$, we define $U_{j}$ to be the set of vertices $u' \in V_r \setminus U'_{j-1}$ such that $|\widetilde{N}_{0} (u') \cap U'_{j-1}| \ge 2^{-r-2} \gamma n $, where $U'_{j'}$ is the set $\bigcup_{0\le i \le j'} U_i$.
By an induction on $j$, Lemma~\ref{lma:bridge} and Proposition~\ref{prp:i-closedadditive}, we deduce that $H[U'_j]$ is $(j+i'+1)$-closed in~$H$.

Let $j_0$ be the smallest integer such that $|U_{j_0}| < 2^{-r-2} \gamma n$.
Since $U_0, U_1, \dots, U_{j_0}$ are disjoint sets, $1 \le j_0 \le \lceil 2^{r+2} / \gamma \rceil$.
Set $W_{r+1} = U_{j_0}'$.
So $W_{r+1}$ is $(j_0+i'+1)$-closed in $H$.
Note that $|W_{r+1}| \ge |U| \ge (2^{i'}d + \gamma/2)n$ by Claim~\ref{clm:U}.
Since $j_0+i'+1 \le i_0$, $W_{r+1}$ is $i_0$-closed in $H$ by Proposition~\ref{prp:i-closedadditive}.
If $W_{r+1} = V_r$, then we are done.
Thus, we may assume that $W_{r+1} \ne V_r$.
Set $V_{r+1} = V_r \setminus W_{r+1}$.
For every $w \in V_{r+1}$,  we have 
\begin{align*}
	|\widetilde{N}_{0}(w) \cap V_{r+1}| & = |\widetilde{N}_{0}(w) \cap V_{r}|  - |\widetilde{N}_{0}(w) \cap U_{j_0-1}'|  - |U_{j_0}| \\
& \ge  (d + (2 - \sum_{1 \le s \le r} 2^{-s})\gamma)n - 2^{-r-2}\gamma n -  2^{-r-2} \gamma n\\
&	=  (d + (2 - \sum_{1 \le s \le r+1} 2^{-s})\gamma)n.
\end{align*}
This completes the proof of the lemma.
\end{proof}

Let $W_1, \dots, W_p$ be a partition of $V(H)$ given by Lemma~\ref{lma:connectioncomponent}.
We may assume $p \ge 2$, or else Lemma~\ref{lma:(K^k_t,i,eta)close} holds.
Set $X :=W_1$ and $Y:= W_2 \cup \dots \cup W_p$.
The next lemma shows that there are many $(X,Y)$-bridges of length~$1$ with respect to~$K^k_t$.
Its proof is rather long and involved, so we postpone the proof to Section~\ref{sec:A_iB_j}.
We would like to point out that the function $\beta(k,t)$ is defined so that this lemma holds.

\begin{lma} \label{lma:A_iB_j}
Let $3 \le k < t$ be integers and $0 < \gamma < \beta(k,t)$.
Then, there exist an integer $n_2 = n_2(k,t, \gamma)$ and a constant $\eps_2 =\eps_2(k,t,\gamma)>0$ such that the following holds.
Let $H$ be a $k$-graph of order $n \ge n_2$ with $\delta_{k-1}(H) \ge (1- \beta(k,t)+\gamma)n$.
Let $X$ and $Y$ be a partition of $V(H)$ with $|X|, |Y| \ge (d(k,t)+ \gamma )n$.
Then, the number of $(X,Y)$-bridges of length~$1$ with respect to~$K^k_t$ in $H$ is at least $\eps_2 n^{t+1}$.
\end{lma}

Assuming Lemma~\ref{lma:A_iB_j}, we are now ready to prove Lemma~\ref{lma:(K^k_t,i,eta)close}.

\begin{proof}[Proof of  Lemma~\ref{lma:(K^k_t,i,eta)close}]
Write $\beta$ and $d$ to be $\beta(k,t)$ and $d(k,t)$ respectively.
Let $i_1'$ and $\eta_1'$ be the constants given by Lemma~\ref{lma:connectioncomponent}.
Let $p_1 = \lfloor ( d+ \gamma/2)^{-1} \rfloor$.
Define $\eta_2', \eta_3', \dots, \eta_{p_1}>0$ to be sufficiently small constants, whose values will be become clear.
Throughout this proof, $n_0$ is assumed to a sufficiently large integer.
Let $H$ be a $k$-graph of order $n \ge n_0$ with $\delta_{k-1}(H) \ge \left(1- \beta+\gamma \right)n$. 
By Lemma~\ref{lma:connectioncomponent}, there exists a vertex partition of $V(H)$ into $W_1, \dots, W_p$ with $1 \le p \le p_1$ such that, for each $1 \le j \le p$, $|W_j| \ge (d+\gamma/2) n$ and $W_j$ is $(K_{t}^k, i_1', \eta_1')$-closed in~$H$.
For $1 \le j \le p$, we say that $i$-close to mean $(K_{t}^k, j(i_1'+1)-1, \eta_j')$-close.
Hence, each $W_j$ is $1$-closed in $H$.
We are going to show, by relabelling the $W_i$'s if necessary, that $W_1 \cup \dots \cup W_j$ is $j$-closed for all $1 \le j \le p$.
We proceed by induction of $j$.
This is true for $j=1$.
Suppose that we have already showed that that for $j < p$, $ W_1 \cup \dots \cup W_j$ is $j$-closed in~$H$.
Set $X = W_1 \cup \dots \cup W_j$ and $Y = W_{j+1} \cup \dots \cup W_p$.
So $|X|,|Y| \ge (d+\gamma/2)n$ and $X,Y$ is a partition of $V(H)$.
By Lemma~\ref{lma:A_iB_j} taking $\gamma = \gamma/2$, the number of $(X,Y)$-bridges of length~$1$ with respect to $K_t^k$ is at least $\eps_2 n^{t+1}$.
Since $Y = W_{j+1} \cup \dots \cup W_p$, by an averaging argument we may assume that the number of $(X,W_{j+1})$-bridges of length~$1$ with respect to $K_t^k$ is at least $\eps_2 n^{t+1}/p$.
Recall that $X$ is $j$-closed and $W_{j+1}$ is $1$-closed.
By Lemma~\ref{lma:bridge}, $X \cup W_{j'} $ is $(j+1)$-closed.
Therefore $V(H) = W_1 \cup \dots \cup W_p$ is $p$-closed as claimed.
Hence, by Proposition~\ref{prp:i-closedadditive}, $H$ is $p_1$-closed.
\end{proof}

\subsection{Proof of Lemma~\ref{lma:A_iB_j}} \label{sec:A_iB_j}

Let $H,X,Y$ be as defined in Lemma~\ref{lma:A_iB_j}.
Here, by a bridge we will mean an $(X,Y)$-bridge of length~$1$ with respect to $K_{t}^k$.
Our aim is to construct many bridges.
The proof of Lemma~\ref{lma:A_iB_j} is broken down into the following steps.
\begin{itemize}
	\item[\rm (i)] There are many copies of $T$ of $K^k_{t-1}$ satisfying one of the following:
		\begin{itemize}
			\item[\rm (a)] $|T \cap X| \le (t-1)/2$ and $|L(T) \cap X| \ge \eps n$ or
			\item[\rm (b)] $|T \cap Y| \le (t-1)/2$ and $|L(T) \cap Y| \ge \eps n$.
		\end{itemize}
		\item[\rm (ii)] For each $T$ satisfying (i), there exists $(\eps n)^2$ bridges $(x,y,T')$ with $T \subseteq T' \cup \{x,y\}$.
\end{itemize}
In the next proposition, we tackle~(i) for $k=3$.

\begin{prp} \label{prp:L(S)bad3}
Let $3 < t$ be an integer and $0 < \gamma < \beta(3,t)$.
Then there exist an integer $n_0 = n_0 (3,t, \gamma)$ and a constant $\eps = \eps(3,t, \gamma)>0$ satisfying the following:
Let $H$ be a $3$-graph of order~$n \ge n_0$ with $\delta_{2}(H) \ge \left(1- \beta(3,t) + \gamma \right)n $.
Let $X$ and $Y$ be a partition of $V(H)$ with $ (d(3,t) + \gamma) n \le |X| \le |Y| $.
Then there exists an integer $0 \le s < t/2-1$ such that there are at least $ \eps n^{t-1}$ copies $T$ of $K_{t-1}^3$ such that either 
\begin{itemize}
			\item[\rm (a)] $|T \cap Y| = s>0$ and  $|L(T) \cap Y| \ge \eps n$ for all $T$, 
			\item[\rm (b)] $|T \cap X| = s>0$ and $|L(T) \cap X| \ge \eps n$ for all $T$,
			\item[\rm (c)] $s =0$, $T \subseteq X$ and $|L(T) \cap Y| \ge \eps n$ for all $T$.
\end{itemize}
\end{prp}

\begin{proof}
We write $\beta$ and $d$ for $\beta(3,t)$ and $d(3,t)$ respectively.
Note that for every distinct $x,x' \in X$, $|N(xx') \cap Y| \ge \delta_2(H) - |X| \ge \gamma n$ as $\beta(3,t) \le 1/2$.
Therefore, there exists at least $\binom{dn}{2} \gamma n$ edges with 2 vertices in~$X$ and one in~$Y$.
Each edge can be greedily extended into a $K_t^k$ in many ways by Proposition~\ref{prp:L(S)}.
Therefore, the number of copies of $K_t^3$ with at least $2$ vertices in $X$ and one vertex in $Y$ is at least 
\begin{align}
	\frac{1}{t!} \binom{dn}{2} \gamma n  \prod_{3 \le i \le t-1} \left( 1 - \binom{i}{k-1} \beta +\gamma \right)n = t c n^t \nonumber
\end{align}
for some constant $c = c(3,t, \gamma) >0$.
By an averaging argument, there exists an integer $1 \le s' \le t-2$ such that there are at least $c n^t $ copies of $K_{t}^k$ with exactly $s'$ vertices in $Y$.
Let $\mathcal{U}$ be the set of $K_t^k$ with $|K_t^k \cap Y| = s'$.
Clearly, $|\mathcal{U}| \ge c n^t $.

Suppose that $1 \le s' \le t/2$.
Let $s = s'-1$.
For each $U \in \mathcal{U}$ there exists a $(t-1)$-set $T \subseteq U$ with $|T \cap Y| = s$.
By an averaging argument, there must exist at least $2 \eps n^t$ $(t-1)$-sets $T$ with $|T \cap Y| = s$
and $T$ is contained in at least $\eps n$ sets $U \in \mathcal{U}$.
Since $|L(T) \cap Y| \ge |\{U \in \mathcal{U}: T \subseteq U\}|$ for all $(t-1)$-sets with $|T \cap Y|=s$, we get either (a) or~(c). 
If $t/2< s' \le t-2$, then $2 \le |U \cap X|  \le t/2$ for all $U \in \mathcal{U}$.
By  a similar argument, we deduce case~(b).
\end{proof}

Next we consider $k \ge 4$.
For $4 \le k <t$, define $l_0(k,t)$ to be the largest integer $l$ such that $\beta(k,t) := \beta(k,t,l)$.
Hence, we have 
\begin{align}
(k-2)/2 & \le l_0(k,t) \le (t-2)/2, \nonumber \\
\beta(k,t) & \le  \frac{2}{\binom{t}{k-1}+ \binom{l_0(k,t)+1}{k-1}+\binom{t-l_0(k,t)-1}{k-1}}, \label{eqn:beta1} \\
\beta(k,t) & \le \frac{1}{\binom{t-1}{k-1} +\binom{l_0(k,t)}{k-1}}
\text{ and }  \beta(k,t) \le  \frac{1}{2 \binom{2 l_0(k,t) +1}{k-1}} . \label{eqn:beta2}
\end{align}
Notice that \eqref{eqn:beta1} is needed for Lemma~\ref{lma:case3} and \eqref{eqn:beta2} is needed for Proposition~\ref{prp:L(S)bad4}.
We now prove (i) when $k \ge 4$.

\begin{prp} \label{prp:L(S)bad4}
Let $4 \le k  < t$ be integers and $0 < \gamma < \beta(k,t)$.
Then there exist an integer $n_0 = n_0 (k,t, \gamma)$ and a constant $\eps = \eps(k,t, \gamma)>0$ satisfying the following:
Let $H$ be a $k$-graph of order~$n \ge n_0$ with $\delta_{k-1}(H) \ge \left(1- \beta(k,t) + \gamma \right)n $.
Let $X$ and $Y$ be a partition of $V(H)$ with $ (d(k,t) + \gamma) n \le |X| \le |Y| $.
Then there exists an integer $l_0(k,t) \le s < t/2-1$ such that there are at least $ \eps n^{t-1}$ copies $T$ of $K_{t-1}^k$ such that either 
\begin{itemize}
			\item[\rm (a)] $|T \cap Y| = s$ and  $|L(T) \cap Y| \ge \eps n$ for all $T$, 
			\item[\rm (b)] $|T \cap X| = s$ and $|L(T) \cap X| \ge \eps n$ for all $T$.
\end{itemize}
\end{prp}

\begin{proof}
We write $l_0$, $\beta$ and $d$ for $l_0(k,t)$, $\beta(k,t)$ and $d(k,t)$ respectively.
Note that $2(l_0 +1) \le t$.
By mimicking the proof of Proposition~\ref{prp:L(S)bad3}, in order to prove this proposition, it suffices to show that there are $c'n^{2(l_0 +1)}$ copies of $K_{2(l_0 +1)}^k$ with exactly $l_0+1$ vertices in each of~$X$ and~$Y$, 
where $c' = c'(k,t, \gamma)>0$.

By Proposition~\ref{prp:L(S)}, for all $1 \le s \le l_0$ and all $s$-sets $S \subseteq V(H)$, we have 
\begin{align*}
|L(S)| + |X| & \ge \left( 1 - \binom{s}{k-1} \beta +\gamma \right)n + (d+ \gamma)n \\
& \ge \left(2 - \left(\binom{l_0}{k-1} + \binom{t-1}{k-1} \right)\beta + 2\gamma \right)  n 
\ge (1 + 2 \gamma)n,
\end{align*}
where the last inequality is due to~\eqref{eqn:beta2}.
Hence, by a greedy algorithm, we can construct a $K_{l_0+1}^k$ in~$X$.
By Proposition~\ref{prp:L(S)}, for all $l_0+1 \le s \le 2l_0 +1$ and all $s$-sets $S\subseteq V(H)$, we have 
\begin{align*}
|L(S)| + |Y| & \ge \left( 1 - \binom{s}{k-1} \beta +\gamma \right)n + n/2 \\
& \ge \left( \frac32 - \binom{2l_0+1}{k-1} \beta +\gamma \right)n  \ge (1 + \gamma)n
\end{align*}
by~\eqref{eqn:beta2}.
Therefore, we can extend the $K_{l_0+1}^k$ in $X$ to a $K_{2(l_0+1)}^k$ with $l_0+1$ vertices in each of~$X$ and~$Y$.
In total, there are at least $(\gamma n )^t/t!$ copies $K_{2(l_0 +1)}^k$ with exactly $l_0+1$ vertices in each of~$X$ and~$Y$.
This completes the proof of the proposition.
\end{proof}

Now we are going to prove~(ii).
Note that $(x,y,T)$ is a bridge if both $T \cup x$ and $T \cup y$ form $K_t^k$ in~$H$.
Given a copy $T$ of $K^k_{t-1}$, $(x,y,T)$ is a bridge for every $x \in X \cap L(T)$ and $y \in Y \cap L(T)$.
We get the following simple proposition, the proof of which we omit.

\begin{prp} \label{prp:case1}
Let $3 \le k < t$ be integers.
Let $H$ be a $k$-graph of order $n$ such that $X$ and $Y$ form a partition of $V(H)$.
If there exists a copy $T$ of $K^k_{t-1}$ in $H$ with $|X \cap L(T)| , |Y \cap L(T)| \ge \eps n$, then the number of $(X,Y)$-bridges $(x,y,T)$ of length~$1$ with respect to~$K_{t}^k$ is at least $(\eps n)^{2}$.
\end{prp}

We now look at the special case when $K^k_{t-1}$ satisfies Proposition~\ref{prp:L(S)bad3}~(c) (and so $k =3$).

\begin{lma} \label{lma:case2}
Let $3 < t$ be an integer and $0 <2 \eps \le  \gamma < \beta(3,t)$.
% Then there exist an integer $n_0 = n_0 (3,t, \gamma)$ and a constant $\eps = \eps(3,t, \gamma)>0$ satisfying the followings:
Let $H$ be a $3$-graph of order~$n \ge n_0$ with $\delta_{2}(H) \ge \left(1- \beta(3,t) + \gamma \right)n $.
Let $X$ and $Y$ be a partition of $V(H)$ with $|X| \le |Y|$.
Let $T$ be a $K_{t-1}^3$ in $H$ such that $T \subseteq X$ and $|L(T) \cap Y| \ge \eps  n$.
Then, the number of $(X,Y)$-bridges $(x,y,T')$ of length~$1$ with respect to $K_{t}^k$
such that $(x,y,T')$ with $T \subseteq T' \cup \{x,y\}$ is at least $(\eps n)^2/(t-1)$.
\end{lma}

\begin{proof}
We write $\delta$, $\beta$ and $d$ for $\delta_2(H)$, $\beta(3,t)$ and $d(3,t)$ respectively.
Let $T = \{v_1, \dots, v_{t-1}\}$.
By Proposition~\ref{prp:case1}, we may assume that $|L(T) \cap X| \le \eps n-1$.
Proposition~\ref{prp:L(S)} implies that $|L(T)| \ge (d+3\gamma)n$.
Hence,
\begin{align*}
|L(T)\cap Y |= |L(T)| - |L(T) \cap X|  \ge (d+\gamma)n - \eps n \ge ( d+\eps )n.
\end{align*}
Pick $z \in L(T)\cap Y$.
For $i \le t-1$, let $T_i = T \setminus  v_{i}$.

First suppose that $|Y \cap L(T_i \cup z )| \ge \eps n$ for some~$i_0$.
Notice that $(v_{i_0}, y, T_{i_0} \cup z)$ is a bridge for every $y \in Y \cap L(T_{i_0} \cup z)$.
Since $z$ is arbitrarily chosen, the lemma holds if the following statement is true:
\begin{align}
|Y \cap L(T_i \cup  z )| & \ge \eps n &\text{for some $1 \le i \le t-1$.} \label{eqn:XXX} 
\end{align}

We now prove~\eqref{eqn:XXX} as follows.
For $1 \le i \le \binom{t-1}{2}$, let $n_i$ be the number of vertices $v$ such that $v$ is in the neighbourhoods of exactly $i$ sets~$S \in \binom{T}{2}$.
Note that
\begin{align}
\sum n_i = n \textrm{ and }\sum i n_i \ge \binom{t-1}{2} \delta \label{eqn:sumn_i2}.
\end{align}
Set $L = ( L(T)\cap Y ) \setminus z$, so $|L| = |L(T)| - |L(T) \cap X| - 1 \ge n_{\binom{t-1}{2}} - \eps n$.
Let $M$ be the set of vertices $v \in V \setminus (T \cup z)$ such that $v$ is in the neighbourhoods of exactly $ \binom{t-1}{2}-1$ sets $S \in \binom{T}{2}$.
This means that $|M| = n_{\binom{t-1}{2}-1}$.
Therefore, \eqref{eqn:sumn_i2} implies that
\begin{align}
	2|L| + |M| + 2 \eps n 
& \ge  2n - \binom{t-1}{2} (n - \delta). \label{eqn:XXX:2L+M}
\end{align}
Let $S = M \cap Y$, so 
\begin{align}
|S| \ge |M| - |X| \ge |M| -n/2. \label{eqn:XXX:M}
\end{align}
Set $S_1 = S \cap L(T_1)$ and for $2 \le i \le t-1$, $S_i = (S \cap L(T_i)) \setminus \bigcup_{j <i} S_j$.
Observe that $S_1, \dots, S_{t-1}$ form a partition of~$S$.
Note that $Y \cap L(T_1 \cup z)$ is the set $S_1 \cap \bigcap_{u \in T_1} N(z u)$.
We may assume that $|Y \cap L(T_1\cup z)|< \eps n$ or else \eqref{eqn:XXX} holds for $i=1$.
This implies that 
\begin{align*}
\eps n & \ge |Y \cap L(T_1\cup z)| = |S_1 \cap \bigcap_{u \in T_1} N(z u)| \\
& \ge  |S_1| -   \sum_{u \in T_1} |S_1 \setminus N(z u)| 
 \ge |S_1| -   \sum_{u \in T} |S_1 \setminus N(z u)|.
\end{align*}
Hence, $\sum_{u \in T} |S_1 \setminus N(z u )| \ge |S_1| - \eps n$ and similar inequalities also hold for $2 \le i \le t-1$.
Since $S = S_1 \cup \dots \cup S_{t-1}$, we have 
\begin{align}
	 \sum_{u \in T} |S \setminus N(z u )| \ge |S| - (t-1)\eps n. \label{eqn:XXX:|Si|}
\end{align}
Recall that $S \subseteq M$ and $M \cap L = \emptyset$.
Therefore, for each $u \in T$,
\begin{align}
	|L \cap N(z u)| & = |L| - |L \setminus N(zu)| 
\ge |L| - |V(H) \setminus N(zu)| + |S\setminus N(z u)| \nonumber \\
& \ge |L| +\delta - n  + |S\setminus N(z u)| . \label{eqn:|LN(zv_j)|}
\end{align}
Suppose that there exists a vertex $y \in L$ is in $N(zu)$ for all but one $u \in T$.
Without loss of generality, $y \in N(zv_i)$ for $1 \le i \le t-2$. 
Then $T_{t-1}\cup \{ z,y\}$ forms a $K^3_t$ in~$H$ and so $y \in L \cap L( T_{t-1} \cup z )$.
This means that 
\begin{align}
	  & 2 |L \cap \bigcup_{1 \le j \le t-1} L( T_i \cup z )| \ge  \sum_{u \in T } { |L \cap N(z u)|} - \left( t-3 \right) |L| \nonumber \\
	   & \overset{\eqref{eqn:|LN(zv_j)|}}{\ge}  2|L| - (t-1)(n-\delta) +   \sum_{u \in T} |S \setminus N(z u)| \nonumber \\ 
	 & \overset{\eqref{eqn:XXX:|Si|}}{\ge}  2|L| - (t-1)( n-\delta ) + |S| - (t-1) \eps n \nonumber \\
 	& \overset{\eqref{eqn:XXX:M}}{\ge} 2|L| - (t-1)(n-\delta) + |M|- n/2  - (t-1) \eps n\nonumber\\
	& \overset{\eqref{eqn:XXX:2L+M}}{\ge} (3/2 - (t+2) \eps ) n - \binom{t}{2} (n- \delta) \nonumber 
	{\ge} 2(t-1) \eps n,
\end{align}
where the last inequality holds, as $n- \delta = (\beta - \delta)n$ and $\beta = 2/(t^2 -3t+4) \le 3/(t^2-t)$. 
Hence, $|Y \cap L(T_i \cup z )| \ge |L \cap L(T_i \cup z)| \ge \eps n$ for some $1 \le i \le t-1$.
Therefore \eqref{eqn:XXX} holds completing the proof of the lemma.
\end{proof}

Next, we are going to consider the cases (a) and (b) in Proposition~\ref{prp:L(S)bad3} and Proposition~\ref{prp:L(S)bad4}.
Here, we set $l_0(3,t) = 1$, so $\beta(3,t)$ also satisfies~\eqref{eqn:beta1}.
The proof is based on proof of Lemma~\ref{lma:case2}.
We would like to point out the lemma below does not assume $|X| \le |Y|$.
We will also need the following inequality, for $s \le  (t-1)/2$, 
\begin{align}
\binom{s}{k-1} +\binom{t-s}{k-1}  \ge \binom{s+1}{k-1} +\binom{t-s-1}{k-1}. \label{eqn:sl}
\end{align}

\begin{lma} \label{lma:case3}
Let $3 \le k < t$ be integers and $0 <\eps \le  \gamma < \beta(k,t)$.
Let $H$ be a $k$-graph of order~$n \ge n_0$ with $\delta_{k-1}(H) \ge \left(1- \beta(k,t) + \gamma \right)n $.
Let $X$ and $Y$ be a partition of $V(H)$ with $|X|, |Y|$.
Let $T$ be a $K_{t-1}^k$ in $H$ such that $l_0(k,t) \le |T \cap Y| \le t/2-1$ and $|L(T) \cap Y| \ge \eps  n$.
Then the number of $(X,Y)$-bridges $(x,y,T')$ of length~$1$ with respect to $K_{t}^k$
such that $(x,y,T')$ with $T \subseteq T' \cup \{x,y\}$ is at least $\eps^2 n^2/t$.
\end{lma}

\begin{proof}
We write $\delta$, $l_0$, $\beta$ and $d$ for $\delta_{k-1}(H)$, $l_0(k,t)$, $\beta(k,t)$ and $d(k,t)$ respectively.
Let $T = \{v_1, \dots, v_{t-1}\}$ such that
\begin{align}
T_X = \{v_1, \dots, v_r\} \subseteq X \textrm{ and } T_Y = \{v_{r+1}, \dots, v_{r+s}\} \subseteq Y. \nonumber
\end{align}
So $r+s+1=t$ and  $l_0 \le s\le t/2-1$.
By Proposition~\ref{prp:case1}, we may further assume that $|L(T) \cap X| \le \eps n-1$.
Proposition~\ref{prp:L(S)} implies that $|L(T)\cap Y| \ge  |L(T)| - |L(T) \cap X| \ge \eps n$.
Pick $z \in L(T) \cap Y$.
For $i \le t-1$, let $T_i = T \setminus  v_{i}$.

First assume that  one of the following statements holds:
\begin{itemize}
\item[(a)] $|Y \cap L(T_{i} \cup z)| \ge \eps n$ for some $1 \le i \le r$, or
\item[(b)] $|X \cap L(T_{r+j}\cup z)| \ge \eps n$ for some $1 \le j \le s$.
\end{itemize}
If (a) holds say for $i =1$, then $(v_1, y, T_1 \cup z)$ is a bridge for every $y \in Y \cap L(T_1 \cup z)$.
Similarly, if (b) holds say for $j=1$, then $(x, v_{r+1}, T_{r+1} \cup z)$ is a bridge for every $x \in X \cap L(T_{r+1} \cup z)$.
Since $z$ is chosen arbitrarily, the lemma holds provided that we can prove (a) or (b) holds.

For $1 \le i \le \binom{t-1}{k-1}$, let $n_i$ be the number of vertices $v$ such that $v$ is in the neighbourhoods of exactly $i$ sets of~$S \in \binom{T}{k-1}$.
Note that
\begin{align}
\sum n_i = n \textrm{ and }\sum i n_i \ge \binom{t-1}{k-1} \delta \label{eqn:sumn_i}.
\end{align}
Set $L = ( L(T)\cap Y ) \setminus z$, so $|L| = |L(T)| - |L(T) \cap X| - 1 \ge n_{\binom{t-1}{k-1}} - \eps n$.
Let $M$ be the set of vertices $v \in V \setminus (T \cup z)$ such that $v$ is in the neighbourhoods of exactly $ \binom{t-1}{k-1}-1$ sets $S \in \binom{T}{k-1}$.
This means that $|M| = n_{\binom{t-1}{k-1}-1}$.
Therefore, \eqref{eqn:sumn_i} implies that
\begin{align}
	2|L| + |M| + 2 \eps n & \ge 2 n_{\binom{t-1}{k-1}} +n_{\binom{t-1}{k-1}-1} \ge
  \sum i n_i -\left( \binom{t-1}{k-1}-2\right) \sum n_i \nonumber
\\
& \ge  2n - \binom{t-1}{k-1} (n - \delta). \label{eqn:2L+M}
\end{align}
Define $R =  M \cap X \cap L(T_X)$ and $S= M \cap Y \cap L(T_Y)$.
Note that $|R| \ge |M \cap X| - \binom{r}{k-1} (n- \delta)$ and a similar inequality for $|S|$.
Thus, we have
\begin{align}
	|R| + |S| & \ge |M| - \left( \binom{r}{k-1} + \binom{s}{k-1} \right) (n- \delta) . \label{eqn:R+S=M}
\end{align}
Set $S_1 = S \cap L(T_1)$ and for $2 \le i \le r$, set $S_i = (S \cap L(T_i)) \setminus \bigcup_{j <i} S_j$.
Observe that $S_1, \dots, S_r$ partition~$S$.
Note that $Y \cap L(T_1 \cup z)$ is the set $S_1 \cap \bigcap_{U \in \binom{T_1}{k-2}} N(z \cup U)$.
We may assume that $|Y \cap L(T_1\cup z)| \le \eps n$ or else (a) holds for $i=1$.
This implies that 
\begin{align*}
\eps n & \ge |Y \cap L(T_1\cup z)| = |S_1 \cap \bigcap_{U \in \binom{T_1}{k-2}} N(z \cup U)| \\
& \ge  |S_1| -   \sum_{U \in \binom{T_1}{k-2}} |S_1 \setminus N(z \cup U )| 
 \ge |S_1| -   \sum_{U \in \binom{T}{k-2}} |S_1 \setminus N(z \cup U )| 
\end{align*}
By a similar argument, we may assume that for all $1 \le i \le r$,
\begin{align}
	 \sum_{U \in \binom{T}{k-2}} |S_i \setminus N(z \cup U )| \ge |S_i| - \eps n. \label{eqn:XYY:|Si|}
\end{align}
Next, we are going to obtain the analogous statement of \eqref{eqn:XYY:|Si|} for~$R$.
Set $R_1 = R \cap L(T_{r+1})$ and for $2 \le j \le s$, $R_j = ( R \cap L(T_{r+j})) \setminus \bigcup_{i <j} R_i$.
Again, $R_1, \dots, R_s$ form a partition of~$R$.
Furthermore, we also may assume that for all $1 \le j \le s$,
\begin{align}
	 \sum_{U \in \binom{T}{k-2}} |R_j \setminus N(z \cup U )| \ge |R_j| - \eps n. \label{eqn:XYY:|Rj|}
\end{align}
or else (b) holds. 
Recall that $\bigcup S_i \cup \bigcup R_j = S \cup R \subseteq M$.
Now we sum \eqref{eqn:XYY:|Si|} overall all $1 \le i \le r$ and sum \eqref{eqn:XYY:|Rj|} overall all $1 \le j \le s$, we get
\begin{align}
	 \sum_{U \in \binom{T}{k-2}} |(S \cup R) \setminus N(z \cup U )| & \ge |R|+|S| - (t-1)\eps n. 
 \label{eqn:XYY:|M|}
\end{align}
Recall that $M \cap L = \emptyset$, so $(S \cup R) \cap L = \emptyset$. 
For any $(k-2)$-sets $U \in \binom{T}{k-2}$,
\begin{align}
	|L \setminus N(z \cup U)| 
% & \le |V(H) \setminus N(z \cup U)| - |M \setminus N(z \cup U)| \nonumber \\
& \le n - \delta - |(S \cup R) \setminus N(z \cup U)|  .
\nonumber
\end{align}
Let $L'$ be the set $L \cap L(T_Y \cup z)$, so
\begin{align}
|L'| & =  |L \cap \bigcap_{U \in \binom{T_Y}{k-2}} N(z \cup U)|
\ge  |L| - \sum_{U \in \binom{T_Y}{k-2}} |L \setminus N(z \cup U)| \nonumber \\
& \ge 
 |L| - \binom{s}{k-2}(n- \delta) + 
 \sum_{U \in \binom{T_Y}{k-2}} |(S \cup R) \setminus N(z \cup U)|.
\label{eqn:|L'|}
\end{align}
In addition, for each $U \in \binom{T}{k-2} \setminus \binom{T_Y}{k-2}$, we have
\begin{align}
	|L' \cap N(z \cup U)| & \ge |L'|  - |L \setminus N(z \cup U)| \nonumber \\
 & \ge |L'| +\delta - n + |(S \cup R) \setminus N(z \cup U)| . \label{eqn:|L'N(zv_j)|}
\end{align}
Next we claim that $ 2 |L \cap \bigcup_{1 \le i \le r} L(T_i \cup z)|$ is at least
\begin{align}
	   \sum_{U \in \binom{T}{k-2} \setminus \binom{T_Y}{k-2}} |L' \cap N(z \cup U)| - \left( \binom{t-1}{k-2} - \binom{s}{k-2}-2 \right) |L'| \label{eqn:XYY:LL}.
\end{align}
Note that $|\binom{T}{k-2} \setminus \binom{T_Y}{k-2}|$ is $\binom{t-1}{k-2} - \binom{s}{k-2}$.
Therefore, \eqref{eqn:XYY:LL} is at most twice the number of vertices $y \in L'$ such that $y$ is in the neighbourhood of $N(z \cup U)$ for all but at most one $U \in \binom{T}{k-2} \setminus \binom{T_Y}{k-2}$.
Let $y$ be such a vertex.
In order to show that \eqref{eqn:XYY:LL} holds, it is enough to show that $y \in L \cap L(T_i \cup z)$ for some $1 \le i \le r$.
Note that $y \in L' \subseteq L$, so it suffices to show that there exists $1 \le i \le r$ such that $y \in N(U \cup z)$ for all $U \in \binom{T_i}{k-2}$.
Let $\mathcal{U}$ be the set of $U \in \binom{T}{k-2}$ such that $y \notin N(z \cup U)$.
If $\mathcal{U} = \emptyset$, then $y \in L(T_1 \cup z)$.
We may assume that $\mathcal{U} \ne \emptyset$.
By the choice of~$y$, at most one $U \in \binom{T}{k-2} \setminus \binom{T_Y}{k-2}$ belongs in $\mathcal{U}$.
Since $y \in L' \subseteq L(T_Y \cup z)$, $\binom{T_Y}{k-2} \cap \mathcal{U} = \emptyset$.
Therefore $\mathcal{U} =\{U\}$ for some $U \in \binom{T}{k-2} \setminus \binom{T_Y}{k-2}$.
Without loss of generality we may assume that $v_1 \in U$.
Therefore, $y \in N(U \cup z)$ for all $U \in \binom{T_1}{k-2}$.
This implies that $y \in L \cap L(T_i \cup z)$ and so \eqref{eqn:XYY:LL} holds.

Finally, we deduce that 
\begin{align*}
& 2 |L \cap \bigcup_{1 \le i \le r} L(T_i )| \\
& \overset{\eqref{eqn:XYY:LL}}{\ge} 
\sum_{U \in \binom{T}{k-2} \setminus \binom{T_Y}{k-2}} |L' \cap N(z \cup U)| - \left( \binom{t-1}{k-2} - \binom{s}{k-2}-2 \right) |L'|\\
& \overset{\eqref{eqn:|L'N(zv_j)|}}{\ge} 
2 |L'| -  \left( \binom{t-1}{k-2} - \binom{s}{k-2} \right) (n-\delta)  + \sum_{U \in \binom{T}{k-2} \setminus \binom{T_Y}{k-2} } |(S \cup R) \setminus N(z \cup U)|\\
& \overset{\eqref{eqn:|L'|}}{\ge }
2 |L| -  \left( \binom{t-1}{k-2} + \binom{s}{k-2} \right) (n-\delta) +  \sum_{U \in \binom{T}{k-2}} |(S \cup R)\setminus N(z \cup U)|\\
& \overset{\eqref{eqn:XYY:|M|}}{\ge } 
2 |L| -  \left( \binom{t-1}{k-2} + \binom{s}{k-2} \right) (n-\delta) + |R|+|S| - (t-1)\eps n\\
& \overset{\eqref{eqn:R+S=M}}{\ge } 
2 |L| + |M| -  \left( \binom{t-1}{k-2} +  \binom{s+1}{k-1} + \binom{r}{k-1} \right) (n-\delta) - (t-1)\eps n\\
& \overset{\eqref{eqn:2L+M}}{\ge } 
(2- (t+1)\eps)n -  \left( \binom{t}{k-1} +  \binom{s+1}{k-1} + \binom{r}{k-1} \right) (n-\delta)\\
& \overset{\eqref{eqn:sl}}{\ge }
(2- (t+1)\eps)n -  \left( \binom{t}{k-1} +  \binom{l_0+1}{k-1} + \binom{t-1-l_0}{k-1} \right) (n-\delta)\\
& \overset{\eqref{eqn:beta1}}{\ge } 
2(t-1)\eps n
\end{align*}
This means that (a) holds for some $1 \le i \le r$.
This completes the proof of lemma.
\end{proof}

Finally we prove Lemma~\ref{lma:A_iB_j}.

\begin{proof}[Proof of  Lemma~\ref{lma:A_iB_j}]
We will only prove the lemma when $k=3$ (since the argument for $k \ge 4$ is immediate by replacing Proposition~\ref{prp:L(S)bad3} with Proposition~\ref{prp:L(S)bad4}).
We write $\delta$, $\beta$ and $d$ for $\delta_2(H)$, $\beta(3,t)$ and $d(3,t)$ respectively.
Without loss of generality, we may assume that $|X| \le |Y|$.
By Proposition~\ref{prp:L(S)bad3}, there exist a constant $\eps = \eps(3,t,\gamma)>0$ and an integer $0 \le s < t/2-1$ such that there exist at least $\eps n^{t-1}$ copies $T$ of $K_{t-1}^3$ such that either 
\begin{itemize}
			\item[\rm (a)] $|T \cap Y| = s>0$ and  $|L(T) \cap Y| \ge \eps n$ for all $T$, 
			\item[\rm (b)] $|T \cap X| = s>0$ and $|L(T) \cap X| \ge \eps n$ for all $T$, or
			\item[\rm (c)] $s =0$, $T \subseteq X$ and $|L(T) \cap Y| \ge \eps n$ for all $T$.
\end{itemize}
Assume that $\eps < 2\gamma$ and set $\eps_2 = \eps^3/((t-1)(t+1)!)$.
If $s=0$ and so (c) holds, then Lemma~\ref{lma:case2} implies that each such $T \subseteq X$ generate at least $(\eps n)^2/(t-1)$ bridges $(x,y, T')$ with $T \subseteq T' \cup \{x,y\}$.
In total, there are at least $\eps^3 n^{t+1}/(t-1)$ bridges with multiplicities at most $(t+1)!$.
Hence, the lemma holds.
If $1 \le s < t/2-1$, then the lemma holds by a similar argument and Lemma~\ref{lma:case3}.
\end{proof}

\section{Remarks on Theorem~\ref{thm:t^k(n,t)} } \label{sec:remarks}

First we prove Corollary~\ref{cor:main} by evaluating $\beta(k,t)$ together with Theorem~\ref{thm:t^k(n,t)}.
\begin{proof}[Proof of Corollary~\ref{cor:main}]
(i) First suppose that $4 \le k < t < 3k/2 - 1$.
If $(k-2)/2 \le l \le (t-2)/2$, then $l+1, t-l-1 <k-1$.
Hence, we have 
\begin{align*}
	\beta(k,t,l) = \min \left\{ \frac{2}{\binom{t}{k-1}} , \frac{1}{\binom{t-1}{k-1}},
\frac{1}{2\binom{2l+1}{k-1}} \right\} 
= \min \left\{ \frac{2}{\binom{t}{k-1}} ,
\frac{1}{2\binom{2l+1}{k-1}} \right\}.
\end{align*}
Therefore, $\beta(k,t) = {2}/{\binom{t}{k-1}}$ as $t > k+1$.
Hence, Corollary~\ref{cor:main}~(i) holds by Theorem~\ref{thm:t^k(n,t)}.

(ii) Suppose that $k \ge 6$ and $2t \ge 3(k-1) + \sqrt{5k^2 -22k+25}$.
Note that, for $(k-2)/2 \le l \le (t-2)/2$,
\begin{align*}
	& \frac{2}{\binom{t}{k-1}+ \binom{l+1}{k-1}+\binom{t-l-1}{k-1} } 
	 \overset{\eqref{eqn:sl}}{\ge}  \frac{2}{\binom{t}{k-1}+ \binom{t- k/2 }{k-1}} \ge \frac{2}{\binom{t}{k-1}+ \binom{t- 3 }{k-1}}
	\ge \binom{t-1}{k-1}^{-1}.
\end{align*}
The last inequality holds because 
\begin{align*}
2 \binom{t-1}{k-1} & = \binom{t}{k-1} + \frac{(t-1)(t-2)(t-2k+2)}{(t-k+1)(t-k)(t-k-1)} \binom{t-3}{k-1}\\
%  = \frac{(t-2k+2)(t-1)(t-2)}{(t-k+1)(t-k)(t-k-1)} \binom{t-3}{k-1}
 & \ge \binom{t}{k-1} + \binom{t-3}{k-1}.
\end{align*} 
Thus, $\beta(k,t,l)$ becomes
\begin{align*}
	\beta(k,t,l) = \min \left\{ \frac{1}{\binom{t-1}{k-1}+ \binom{l}{k-1}},
\frac{1}{2\binom{2l+1}{k-1}} \right\} \le \frac{1}{\binom{t-1}{k-1}}
\end{align*}
and so $\beta(k,t) = \beta(k,t, \lceil(k-2)/2 \rceil ) = \binom{t-1}{k-1}^{-1}$.
Hence, Corollary~\ref{cor:main}~(ii) holds by Theorem~\ref{thm:t^k(n,t)}.
\end{proof}

\section{Closing remarks}

It is likely that the bounds given by Corollary~\ref{cor:main} are not optimal. 
A strengthening of Lemma~\ref{lma:(K^k_t,i,eta)close} would lead to a better bound for Corollary~\ref{cor:main}~(i).
However, more work is required in order to improve the bound given by Corollary~\ref{cor:main}~(ii).
Suppose that $k,t,n$ satisfy Corollary~\ref{cor:main}~(ii).
Recall that in order to find a $K_t^k$-factor in $H$, we also find an almost $K_t^k$-factor in~$H$ using Proposition~\ref{prp:approximate}.
This proposition requires that $\delta_{k-1}(H) \ge (1-\binom{t-1}{k-1}^{-1} + \gamma) n$, which is precisely $(1-\beta(k,t) + \gamma) n$.
Therefore, we need to strengthen both Proposition~\ref{prp:approximate} and Lemma~\ref{lma:(K^k_t,i,eta)close} in order to improve Corollary~\ref{cor:main}~(ii).
Furthermore, the bound $\delta_{k-1}(H) \ge (1-\binom{t-1}{k-1}^{-1} + \gamma) n$ is natural in the sense that this is the threshold guaranteeing that we can greedily extend any vertex $v$ into a~$K_t^k$, see Proposition~\ref{prp:L(S)}.
If $\delta_{k-1}(H) < (1-\binom{t-1}{k-1}^{-1})n$, then many of our arguments would not hold.
Therefore, improving Corollary~\ref{cor:main}~(ii) seems difficult.
As we have mentioned in the introduction, we would like to know the asymptotic values of $t^k(n,t)$.

\section*{Acknowledgment}
The authors would like to thank Jie Han and the anonymous referees for the helpful comments and the careful reviews.

\end{document}